\documentclass[11pt]{article}

\usepackage[a4paper,margin=1.02in]{geometry}
\usepackage{amsmath,amssymb,amsthm,mathtools,mathrsfs}
\usepackage{enumitem}
\usepackage[colorlinks=true,linkcolor=blue,citecolor=blue,urlcolor=blue]{hyperref}

\newtheorem{theorem}{Theorem}[section]
\newtheorem{proposition}[theorem]{Proposition}
\newtheorem{lemma}[theorem]{Lemma}

\theoremstyle{remark}
\newtheorem{remark}[theorem]{Remark}

\newcommand{\R}{\mathbb R}
\newcommand{\dd}{\,\mathrm d}
\newcommand{\Dt}{D_t}
\newcommand{\divg}{\operatorname{div}}
\newcommand{\norm}[1]{\left\lVert#1\right\rVert}
\newcommand{\abs}[1]{\left\lvert#1\right\rvert}
\newcommand{\pos}[1]{\left[#1\right]_+}
\newcommand{\Q}{\mathcal Q}
\newcommand{\Y}{\mathcal Y}
\newcommand{\Scal}{\mathcal S}
\newcommand{\U}{\mathscr U}
\newcommand{\V}{\mathscr V}
\newcommand{\Pcal}{\mathscr P}
\newcommand{\Z}{\mathscr Z}

\title{Sharp sign criterion for global existence and blow-up of strong solutions with arbitrarily large initial data in the 3D spherically symmetric compressible Navier-Stokes equations}
\author{
    \textbf{Xiangdi Huang}$^{a}$\thanks{%
        E-mail addresses:
        xdhuang@amss.ac.cn (X.~Huang);
        bingshuli.math@gmail.com (B.~Li).
    },
    \textbf{Bingshu Li}$^{b}$
    \\[0.5em]
    \parbox{0.98\textwidth}{%
        \centering
        \small
        a. State Key Laboratory of Mathematical Sciences,
        Academy of Mathematics and Systems Science,\\
        Chinese Academy of Sciences,
        Beijing 100190, China;
        \\[0.4em]
        b. School of Mathematics, Jilin University,
        Changchun 130012, China;
    }
}
\date{}

\begin{document}

\maketitle

\begin{abstract}
	We establish a sharp criterion for global-in-time existence versus finite-time blow-up of strong solutions to the 3D spherically symmetric compressible Navier-Stokes equations on a solid ball, based solely on the initial sign of effective velocity. The viscosity coefficients are assumed to satisfy the Bresch-Desjardins structure $\mu=\rho^{\alpha}$, $\lambda=(\alpha-1)\rho^{\alpha}$, with the effective velocity given by $v=u+\alpha \rho^{\alpha-2}\rho_r$. Previously, global existence results for strong solutions in higher dimensions were restricted to the case $\alpha\le 1$ with the endpoint $\alpha=1$ corresponding to the viscous Saint-Venant (shallow water) system. In this paper, we extend the global existence theory beyond this threshold to the supercritical regime $\alpha>1$. Specifically, whenever the initial effective velocity is nonnegative on the boundary of the ball, we prove the global-in-time existence of strong solutions for arbitrarily large initial data. This provides the first result for $\alpha>1$ in the multi-dimensional setting. The main difficulty lies in deriving a uniform lower bound for the density, which requires handling singular estimates at the center of the ball. To overcome this, we exploit several novel quantities near the center and employ a refined maximum principle to establish, for the first time, the Lipschitz continuity of the velocity at the center-thereby resolving the singularity of the velocity gradient. Subsequently,  we derive Dini-Gronwall-type inequalities for suitably constructed functions, which close the density lower bound estimate. This argument crucially relies on the non-negativity of the initial effective velocity on the boundary. Conversely, we can construct a family of initial data for which the initial effective velocity is negative on the boundary, such that the corresponding strong solutions blow up in finite time, with vacuum appearing exactly on the boundary at the blow-up time. Our results establish a clean dichotomy between global regularity and singularity formation, governed purely by the initial sign of the effective velocity.
\end{abstract}

\medskip
{\small
\par}

\clearpage
\setcounter{tocdepth}{2}
\tableofcontents
\bigskip

\section{Introduction}

In this paper, we consider the isentropic compressible Navier--Stokes equations with density dependent viscosity coefficients in the ball $B_R\subset\R^3$:
\begin{equation}
\label{eq:ns-system}
\left\{
\begin{aligned}
 &\partial_t\rho+\divg(\rho\boldsymbol u)=0,\\
 &\partial_t(\rho\boldsymbol u)
   +\divg(\rho\boldsymbol u\otimes\boldsymbol u)
   +\nabla P(\rho)
   -\divg\!\bigl(\mu(\rho)\mathbb D\boldsymbol u\bigr)
   -\nabla\!\bigl(\lambda(\rho)\divg\boldsymbol u\bigr)=0,
\end{aligned}
\right.
\end{equation}
where
\[
 \mathbb D\boldsymbol u
 :=\frac{\nabla\boldsymbol u+(\nabla\boldsymbol u)^{\mathsf T}}{2},
 \qquad
 P(\rho)=\rho^\gamma,
 \qquad
 \mu(\rho)=\rho^\alpha,
 \qquad
 \lambda(\rho)=(\alpha-1)\rho^\alpha.
\]
With the stress convention in \eqref{eq:ns-system}, these coefficients satisfy
the Bresch--Desjardins compatibility relation.  We focus on the regime
$\alpha>1$ and $\gamma>\alpha$ and study spherically symmetric solutions on
$B_R$ with the no-slip boundary condition
\[
\boldsymbol u(t,x) = 0, \quad \text{ on } \partial B_R.
\]

For a spherically symmetric solution, write
\[
 \rho(t,x)=\rho(t,r),
 \qquad
 \boldsymbol u(t,x)=u(t,r)\frac{x}{r},
 \qquad
 r=|x|\in[0,R].
\]
Then \eqref{eq:ns-system} takes the radial form
\begin{equation}
\label{eq:radial-system}
\left\{
\begin{aligned}
 &\rho_t+(\rho u)_r+\frac{2}{r}\rho u=0,\\
 &(\rho u)_t+(\rho u^2)_r+\frac{2}{r}\rho u^2
   +(\rho^\gamma)_r
   -\left[\frac{\alpha}{r^2}\rho^\alpha(r^2u)_r\right]_r
   +\frac{2}{r}(\rho^\alpha)_r u=0.
\end{aligned}
\right.
\end{equation}

Vacuum is one of the main obstructions to classical solvability for compressible viscous flow.  If the viscosities vanish with the density, the approach to $\rho=0$ simultaneously degenerates the parabolic smoothing and weakens the available control of the velocity.  The central problem is whether, in the supercritical regime $\alpha>1$, an initially smooth and strictly positive density remains uniformly bounded away from zero or can instead develop vacuum somewhere in the ball.

The classical constant viscosity theory provides the broad background. Matsumura and Nishida established global smooth solutions for small perturbations of a positive equilibrium \cite{MatsumuraNishida1980}, while Lions \cite{Lions1998} and Feireisl--Novotn\'y--Petzeltov\'a \cite{FeireislNovotnyPetzeltova2001} developed the global finite energy weak theory for large data and possible vacuum. For density dependent viscosities, the decisive structural development was the additional entropy discovered by Bresch and Desjardins \cite{BreschDesjardins2003} and developed with Lin \cite{BreschDesjardinsLin2003}.  The associated modified velocity couples the physical velocity to a density gradient and yields estimates unavailable for generic viscosity laws.

This BD structure underlies a substantial theory of weak solutions.  Among them, Mellet and Vasseur obtained stability estimates \cite{MelletVasseur2007}, Guo--Jiu--Xin constructed global spherically symmetric weak solutions \cite{GuoJiuXin2008}, Li--Xin \cite{LiXin2015} and
Vasseur--Yu \cite{VasseurYu2016} treated multidimensional degenerate systems, and Bresch--Vasseur--Yu later covered the physical symmetric stress and broad nonlinear viscosity laws \cite{BreschVasseurYu2022}. In particular, the latter theory includes power laws on both sides of $\alpha=1$.  Global weak existence is therefore not the issue addressed here: weak solutions may contain vacuum and do not decide whether an initially positive classical branch stays uniformly separated from zero.

One-dimensional results revealed a complementary role of the effective variable in strong theory.  Haspot obtained global strong solutions in a subcritical range \cite{Haspot2018}, while Constantin--Drivas--Nguyen--Pasqualotto \cite{ConstantinEtAl2020} and Burtea--Haspot \cite{BurteaHaspot2020} obtained results extending into $\alpha>1$ under one-sided or Oleinik-type hypotheses on an effective quantity.  The one-dimensional problem, however, has neither the coordinate singularity at the center of a ball nor the constraint of a boundary at finite radius present here.  We also distinguish our direction of vacuum dynamics from that of Li--Li--Xin \cite{LiLiXin2008}: their entropy weak solutions may start with vacuum which subsequently disappears, whereas our second result starts from a strictly positive classical state and produces vacuum on a fixed boundary.

Spherical symmetry reduces the system to one radial variable but does not turn it into a nonsingular one-dimensional problem.  The weights $r^2$, the quotient $u/r$, and the difference $u_r-u/r$ make the origin a genuine analytic difficulty.  Earlier spherical vacuum results often locate the exceptional behavior at this center.  Xin--Yuan proved that vacuum does not form away from the origin under their noninteraction assumptions and also gave a separate criterion on velocity regularity excluding vacuum formation \cite{XinYuan2006}, and Guo--Wang--Wang obtained a spherical theory for large data and another viscosity law in which positivity is protected away from $r=0$ \cite{GuoWangWang2018}. By contrast, the loss of density constructed below occurs on the outer boundary and nowhere else.

The closest related result is the recent classical theory for large data in the same BD problem on a fixed ball with no slip.  Zhang proved global radial strong solutions for arbitrary strictly positive data in the ranges
\[
 N=2:\quad \frac45\leq\alpha<1,\quad \gamma>1,
 \qquad
 N=3:\quad \frac78\leq\alpha<1,\quad 1<\gamma<9\alpha-6
\]
\cite{Zhang2025}.  Huang--Meng--Zhang \cite{HuangMengZhang2025} enlarged this to
\[
 N=2:\quad \frac12<\alpha<1,\quad \gamma>1,
\]
and, with the number $n_3(\alpha)$ defined in their paper, to
\[
 N=3:\quad 0.686<\alpha<1,
 \qquad
 1<\gamma<6\alpha-3+\frac{3-5\alpha}{2n_3(\alpha)}.
\]
They also treated a two-dimensional endpoint at $\alpha=1$ under an additional pressure restriction and constructed global radial weak solutions allowing vacuum in portions of the range $\alpha\geq1$.  Lei subsequently obtained classical solutions and proved that vacuum does not form in finite time for
\[
 N=2:\quad 0.54369<\alpha<1,\quad \gamma>1,
 \qquad
 N=3:\quad 0.67661<\alpha<1,\quad 1<\gamma<6\alpha-3
\]
\cite{Lei2026}.  Chen--Zhang--Zhu proved large radial regular solutions with vacuum in the far field for $1/2<\alpha<1$ \cite{ChenZhangZhu2026}, and X.~Zhang studied a three-dimensional isothermal exterior problem with vacuum in the far field \cite{XingyangZhang2026}.  These domains and states in the far field differ from the positive problem with a fixed boundary considered here, but the results make the threshold in the existing unrestricted radial classical theory clear.

There are also multidimensional strong results for $\alpha>1$ under other
structures.  Yu \cite{Yu2023} and Huang--Li--Zhang \cite{HuangLiZhang2026} treated the three-dimensional Cauchy problem around a sufficiently large positive density level; Xin--Zhu \cite{XinZhu2021} studied a special class of expanding velocities with vacuum at infinity; and Du--Guo--Xu \cite{DuGuoXu2026} proved a result on a bounded domain with slip boundary
conditions, a large external potential, and a sufficiently large density lower bound. Constructions with moving interfaces and free boundaries constitute another neighboring theory: Guo--Xu--Zhang \cite{GuoXuZhang2025} splice positive interior and exterior solutions across a moving stress-free interface, while Chen--Zhang--Zhu \cite{ChenZhangZhuFree2026} treat a broader vacuum class based on distance that includes the physical vacuum free boundary.  Finally, the implosion construction of Merle--Rapha\"el--Rodnianski--Szeftel \cite{MerleEtAl2022} and its degenerate viscosity counterpart due to Chen--Liu--Zhu \cite{ChenLiuZhu2026} concern blowup of the density at the origin rather than the formation of vacuum on the boundary.

We next explain why $\alpha=1$ is a structural threshold for excluding vacuum. Define the BD potential $\Phi_\alpha$ by
\begin{equation}
 \label{eq:intro-BD-potential}
 \Phi_\alpha'(\rho)=\alpha\rho^{\alpha-2},
 \qquad
 \Phi_\alpha(\rho)=
 \begin{cases}
 \displaystyle
 \frac{\alpha}{\alpha-1}\bigl(\rho^{\alpha-1}-1\bigr),
 &\alpha\ne1,\\[6pt]
 \log\rho,&\alpha=1.
 \end{cases}
\end{equation}
The radial effective velocity then satisfies
\begin{equation}
 \label{eq:intro-effective-velocity}
 v-u=\partial_r\Phi_\alpha(\rho).
\end{equation}
Thus suitable estimates for the physical and effective velocities control the spatial variation of a transformed density variable.

For $\alpha\le1$, one has
\[
 \Phi_\alpha(\rho)\longrightarrow-\infty
 \qquad\text{as }\rho\downarrow0.
\]
In this transformed variable, vacuum is therefore infinitely far from every
positive density value.  Since conservation of positive mass guarantees that the density attains at least its positive spatial average somewhere in the ball, sufficiently strong control of the spatial variation of this singular variable prevents the density from reaching zero across a bounded domain.  In the existing $\alpha<1$ radial arguments, this principle is realized through estimates for nearby singular negative powers of the density; at the two-dimensional endpoint $\alpha=1$, it is realized through control of $\partial_r\log\rho=v-u$; see \cite[Sections~3.3--3.4 and 4.2]{HuangMengZhang2025}.  For $\alpha>1$, by contrast, $\Phi_\alpha$ remains finite at $\rho=0$, so this mechanism for a lower bound disappears.

The same threshold is reflected in the scale of viscous diffusion per unit
mass,
\[
 \frac{\mu(\rho)}{\rho}=\rho^{\alpha-1}.
\]
It becomes stronger near vacuum for $\alpha<1$, remains of order one for
$\alpha=1$, and vanishes for $\alpha>1$.  Hence, in the supercritical regime, the singular density barrier is lost precisely where the viscous regularization becomes weakest.

The boundary trace explains which profile in mass coordinates is relevant.  For every positive radial classical solution, the effective equation derived below is
\begin{equation}
 \label{eq:intro-effective-equation}
 \Dt v+F_1(\rho)(v-u)=0,
 \qquad
 F_1(\rho)=\frac\gamma\alpha\rho^{\gamma-\alpha}.
\end{equation}
Since $u(t,R)=0$, the fixed boundary is a material trajectory and
\begin{equation}
 \label{eq:intro-wall-ode}
 \partial_t v(t,R)+F_1(\rho(t,R))v(t,R)=0,
 \qquad
 v(t,R)=v_0(R)
 \exp\!\left(-\int_0^tF_1(\rho(s,R))\dd s\right).
\end{equation}
Thus the sign of $v(t,R)$ is preserved throughout the positive classical
lifespan.

To see the corresponding density geometry, introduce
\[
 \delta(t,r):=\int_r^R\rho(t,s)s^2\dd s.
\]
Thus $\delta$ is the mass contained between $r$ and the boundary; the systematic formulation in mass coordinates is developed below.  Since $\delta_r=-\rho r^2$, the identity for the effective velocity gives
\[
 (\rho^\alpha)_\delta=-\frac{v-u}{r^2}.
\]
At the no-slip boundary this becomes
\[
 (\rho^\alpha)_\delta(t,0)=-\frac{v(t,R)}{R^2}.
\]
Consequently, a finite nonzero negative boundary trace $v(t,R)<0$ selects a finite positive inward mass slope:
\[
 \rho^\alpha(t,\delta)
 =\rho(t,R)^\alpha+A(t)\delta+o(\delta),
 \qquad
 A(t)=-\frac{v(t,R)}{R^2}>0.
\]
If the boundary density tends to zero while this slope remains nondegenerate, the candidate terminal tangent is therefore
\[
 \rho^\alpha\sim A\delta.
\]

It remains to ask whether such a mass profile can occupy a collar of finite
physical thickness.  A mass slice $\dd\delta$ has radial thickness
\[
 \dd(R-r)=\frac{\dd\delta}{\rho r^2}.
\]
For $\rho^\alpha\sim A\delta$, one has $\rho\sim(A\delta)^{1/\alpha}$, and therefore, near a finite boundary where $r\sim R$,
\[
 R-r\sim \frac{1}{R^2A^{1/\alpha}}
 \int_0^\delta s^{-1/\alpha}\dd s.
\]
When $\alpha>1$,
\[
 \int_0^\delta s^{-1/\alpha}\dd s
 =\frac{\alpha}{\alpha-1}\delta^{(\alpha-1)/\alpha}<\infty.
\]
Equivalently, the corresponding behavior in physical space is
\[
 \rho(r)\sim C(R-r)^{1/(\alpha-1)}.
\]
For $\alpha\le1$, the same integral diverges, so an affine $\rho^\alpha$ profile cannot connect positive interior density to vacuum at a boundary of finite radius. It is precisely the mechanism of finite capacity exploited by the construction in the later sections.

Our first theorem uses the favorable sign.  For $\alpha>1$ and
$\gamma>\alpha$, if a positive radial classical solution has bounded density above on a finite time interval and $v_0(R)\geq0$, then the density is bounded away from zero on the entire ball during that interval.  The proof combines the boundary condition with auxiliary quantities compatible with the center, a refined maximum principle for reciprocal density variables, and Dini--Gronwall estimates.  In the exponent range specified below, an upper bound for the density is available on every bounded time interval.  The conditional lower bound then yields a global positive radial classical solution for every compatible datum with nonnegative boundary trace in that range.

Our construction theorem realizes the opposite sign.  In the same exponent
range \eqref{eq:exponent-range}, we construct smooth compatible radial data
with strictly positive density and $v_{0,k}(R)<0$ whose maximal positive
classical lifespan satisfies $T_k^*\leq Ck^{(\alpha-1)/\alpha}$.  The density converges uniformly at the terminal time to a nontrivial profile whose zero set is exactly $\{R\}$. Consequently, the breakdown is finite-time cavitation at the fixed boundary.

Together, these results identify the initial boundary effective velocity as a boundary order parameter for the supercritical radial BD problem.  They show two rigorously opposite mechanisms for the boundary density selected by the same trace whose sign is preserved.  Note that the result for negative boundary trace is an explicit construction, not a classification of all data with $v_0(R)<0$.

The paper is organized as follows.  Section~2 states the main results and records their logical dependencies.  Section~3 fixes the solution framework, mass coordinates, effective variables, and local continuation principle.  Section~4 establishes the estimates on bounded time intervals used both in the global corollary for nonnegative boundary trace and in the construction.  Section~5 proves the lower density theorem and the global corollary.  Section~6 constructs the family with negative boundary trace, proves the endpoint by finite capacity, and localizes the terminal vacuum on the boundary.

\section{Main results}

Throughout the paper, a \emph{positive radial classical solution} is a radial solution of \eqref{eq:ns-system} with $\rho>0$ whose density and velocity have the regularity stated in \eqref{eq:local-rho-class}--\eqref{eq:local-u-class} on every compact time
interval in their lifespan.  Smooth radial parity is understood at the
center; in particular, $\rho_r(t,0)=0$ and $u(t,0)=0$.  At the outer boundary we impose $u(t,R)=0$.  Initial data are called \emph{compatible} when they satisfy the condition in Proposition~\ref{prop:local-restart}.

The parameter range used for the unconditional global corollary and for the construction of boundary vacuum is
\begin{equation}
	\label{eq:exponent-range}
	1<\alpha<11.7,
	\qquad
	2\alpha-1\le\gamma<3\alpha-1.
\end{equation}
The upper density and continuation estimates on bounded time intervals used below are available throughout \eqref{eq:exponent-range} from the corresponding higher-dimensional radial theory of Huang--Meng--Zhang \cite{HuangMengZhang2025}.

We now introduce the two main results of this paper.

\begin{theorem}[Lower density bound for nonnegative boundary trace]
	\label{thm:main}
	Let $\alpha>1$ and $\gamma>\alpha$. Let $(\rho,u)$ be a
	maximal positive radial classical solution of \eqref{eq:ns-system}
	on $[0,T)\times B_R$, where $T\in(0,\infty)$ is its maximal
	existence time in the class of positive radial classical solutions.
	Assume that
	\begin{equation}
		\label{eq:upper-density}
		\sup_{0\le t<T}\|\rho(t)\|_{L^\infty(B_R)}
		=: \overline\rho_T<\infty,
	\end{equation}
	and that the initial effective velocity has nonnegative
	boundary trace:
	\[
		v_0(R)\ge0.
	\]
	Then there exists a constant $c_T>0$, depending only on
	$R,\alpha,\gamma,T,\overline\rho_T$ and the initial data,
	such that
	\begin{equation}
		\label{eq:main-conclusion}
		\rho(t,r)\ge c_T
		\qquad
		(0\le t<T,\ 0\le r\le R).
	\end{equation}
\end{theorem}

\begin{remark}
The conclusion of Theorem~\ref{thm:main} is a lower bound that remains uniform as $t$ approaches the maximal existence time $T$. Such a bound does not follow merely from classical positivity: positivity on $[0,T)$ only guarantees a positive lower bound on each compact time interval $[0,\tau]$ with $\tau<T$.

For compatible initial data in the exponent range \eqref{eq:exponent-range}, this estimate yields global existence. Indeed, suppose that the maximal existence time $T$ were finite. Proposition~\ref{prop:bounded-time} supplies a uniform upper density bound on $[0,T)$, and the condition $v_0(R)\ge0$ allows Theorem~\ref{thm:main} to give
$$
0<c_T\le\rho(t,r)\le\overline\rho_T<\infty
\qquad (0\le t<T,\ 0\le r\le R).
$$
The continuation criterion in Proposition~\ref{prop:bounded-time} then permits the positive classical solution to extend beyond $T$, contradicting maximality. Hence the maximal existence time is infinite,
as stated in Corollary~\ref{cor:global-good-wall}.

The resulting density bounds are uniform on every finite
time interval; the constants may depend on the time horizon.
\end{remark}

As a direct application of the theorem, we have the following corollary.

\begin{theorem}[Global solutions for nonnegative boundary trace]
	\label{cor:global-good-wall}
	Assume \eqref{eq:exponent-range} holds.  Let $(\rho_0,\boldsymbol u_0)$ be compatible radial initial data satisfying the hypotheses of Proposition~\ref{prop:local-restart}, with $\rho_0$ strictly positive, and suppose that $v_0(R)\ge0$.  Then the corresponding maximal positive radial classical solution is global:
	\[
	T^*=\infty.
	\]
	Moreover, for every finite $T_0>0$, there are constants $0<c_{T_0}\le C_{T_0}<\infty$ such that
	\[
	c_{T_0}\le\rho(t,r)\le C_{T_0}
	\qquad
	(0\le t\le T_0,\ 0\le r\le R).
	\]
\end{theorem}
\begin{remark}
The significance of this theorem lies in establishing the following fact: for the supercritical case $\alpha>1$, if the initial effective velocity is nonnegative on the boundary, then the three-dimensional compressible Navier–Stokes equations admit a unique global strong solution. This is the first global existence result for strong solutions with arbitrarily large initial data in the multi-dimensional setting. In contrast to Haspot's one-dimensional result~\cite{BurteaHaspot2020}, the higher-dimensional system introduces additional geometric singular terms that pose substantial difficulties. The main difficulty lies in establishing a uniform lower bound for the density, which necessitates handling singular estimates of the velocity at the center of the ball-a singularity that does not arise in the one-dimensional setting. To overcome this, we exploit several novel quantities near the center and employ a refined maximum principle to establish, for the first time, the Lipschitz continuity of the velocity at the center-thereby resolving the singularity of the velocity gradient. Subsequently,  we derive Dini-Gronwall-type inequalities for suitably constructed functions, which close the density lower bound estimate. This argument crucially relies on the non-negativity of the initial effective velocity on the boundary.
\end{remark}

Before turning to the opposite boundary sign, we record a localization estimate which does not impose any sign condition on the effective velocity on the boundary.

\begin{theorem}[Interior density preservation and vacuum localization]
	\label{thm:wall-independent}
	Let $\alpha>1$, $\gamma>\alpha$, and $0<T<\infty$.  Let $(\rho,u)$ be a positive radial classical solution on $[0,T)\times B_R$ satisfying the upper density bound \eqref{eq:upper-density}.  Let $m$ and $M$ be the radial mass coordinate and conserved total mass defined in \eqref{eq:mass-coordinate}.  No sign condition is imposed on $v_0(R)$.  Then there exists a finite constant $C_X$, depending only on $R,\alpha,\gamma,T,\overline\rho_T$ and the initial data, such that
	\begin{equation}
		\label{eq:weighted-reciprocal-bound}
		\frac{(M-m)^{1/\alpha}}{\rho(t,m)}\le C_X
		\qquad
		(0\le t<T,\ 0\le m\le M).
	\end{equation}
	Consequently, for every $\varepsilon\in(0,R)$ there exists $c_{T,\varepsilon}>0$, with the same constant dependence together with $\varepsilon$, such that
	\begin{equation}
		\label{eq:wall-independent-interior-lower}
		\rho(t,r)\ge c_{T,\varepsilon}
		\qquad
		(0\le t<T,\ 0\le r\le R-\varepsilon).
	\end{equation}
	In particular, if $t_j\uparrow T$ and $\rho(t_j,r_j)\to0$, then $r_j\to R$.  If, in addition, $\rho(t,\cdot)\to\rho_*$ uniformly on $[0,R]$ as $t\uparrow T$ and $\rho_*$ has a zero, then
	\begin{equation}
		\label{eq:wall-independent-zero-set}
		\{r\in[0,R]:\rho_*(r)=0\}=\{R\}.
	\end{equation}
\end{theorem}

\begin{remark}
Under the hypotheses of Theorem~\ref{thm:wall-independent}, the uniform upper bound for the density prevents finite-time vacuum formation in the interior of the ball, including at the center. Indeed, the density remains uniformly bounded away from zero on every fixed smaller ball $B_{R-\varepsilon}$ throughout $[0,T)$, regardless of the sign of the initial boundary effective velocity $v_0(R)$. Consequently, any sequence of points along which the density tends to zero as $t\uparrow T$ must approach the outer boundary $\partial B_R$. If a uniform terminal density profile exists and contains vacuum, its vacuum set is exactly $\partial B_R$. The theorem therefore identifies the only possible location of finite-time vacuum formation.
\end{remark}

The second result accounts for the opposite direction. For large data, we can construct a family of initial data that blow up on the boundary in finite time with negative boundary trace.

\begin{theorem}[Finite time boundary vacuum for data with negative boundary trace]
	\label{thm:bad-wall}
	Assume \eqref{eq:exponent-range}.  Then there exist $k_0,C>0$ such that, for every $0<k<k_0$, one can choose smooth compatible radial initial data with strictly positive density and
	\[
	v_{0,k}(R)<0.
	\]
	The corresponding maximal positive classical solution exists on $[0,T_k^*)$, where
	\[
	0<T_k^*\le Ck^{(\alpha-1)/\alpha}<\infty.
	\]
	Moreover, there is a nonnegative, nontrivial terminal profile $\rho_{k,*}$ such that
	\[
	\rho_k(t)\longrightarrow\rho_{k,*}
	\quad\text{uniformly on }\overline{B_R}
	\quad\text{as }t\uparrow T_k^*,
	\]
	and
	\[
	\{r\in[0,R]:\rho_{k,*}(r)=0\}=\{R\}.
	\]
\end{theorem}

Note there we do not take limit in $k$, $k$ is just an index indicating we have a family of initial data.

\paragraph{Proof architecture.}
The two arguments use the same boundary trace but otherwise have distinct
structures.
\begin{enumerate}[label=\textup{(\roman*)},leftmargin=2.2em]
	\item In the good boundary case, the transverse strain rate functions estimates of Proposition~\ref{prop:angular-bounds} control the effective pressure in Proposition~\ref{prop:Y-bounds}, which closes the maximum inequality for reciprocal density of Proposition~\ref{prop:Q-inequality} by a Dini--Gronwall argument. Without using the boundary sign, their pathwise form also closes the critical weighted reciprocal estimate in Theorem~\ref{thm:wall-independent}.
	\item In the bad boundary case, Proposition~\ref{prop:bounded-time} supplies an upper bound for the density and a criterion for density loss.  The compatible data of Lemma~\ref{lem:data} create a negative flux plateau.  The scaled maximum principle and argument on an observation collar preserve enough of this flux to force the contradiction from finite capacity in Proposition~\ref{prop:finite-capacity}. Proposition~\ref{prop:terminal-wall} then constructs and localizes the terminal density trace.
\end{enumerate}

\begin{remark}[Scope of the sign mechanism]
	Theorem~\ref{thm:main} is conditional on a density upper bound on the bounded time interval under consideration, and its constant $c_T$ need not be uniform as $T\to\infty$. Corollary~\ref{cor:global-good-wall} becomes unconditional only because that upper bound is available in \eqref{eq:exponent-range}.  Conversely, Theorem~\ref{thm:bad-wall} constructs a particular family with negative boundary trace; it does not assert that every datum with $v_0(R)<0$ develops vacuum.
\end{remark}

\section{Preliminaries}
This section records the radial mass coordinates, the effective and transverse strain rate functions, two elementary Dini lemmas, and the local existence and continuation principle under positive density used later.

\subsection{Radial mass coordinates}

The outward radial mass coordinate and total mass are defined as
\begin{equation}
	\label{eq:mass-coordinate}
	m(t,r):=\int_0^r\rho(t,s)s^2\dd s,
	\qquad
	M:=m(t,R).
\end{equation}
The no-slip condition and the continuity equation make $M$ independent of
time.  Since $\rho>0$, the map $r\mapsto m(t,r)$ is strictly increasing for each $t<T$; let $r=r(t,m)$ be its inverse.  Differentiating $m(t,r(t,m))=m$ at fixed $m$ gives
\begin{equation}
	\label{eq:mass-flow}
	r_t=u.
\end{equation}
The continuity equation also gives
\begin{equation}
	m_r=\rho r^2,
	\qquad
	m_t=-\rho u r^2.
\end{equation}

We also use the inward mass coordinate, which is useful when we study the behavior of solution near the boundary.
\begin{equation}
	\delta(t,r):=M-m(t,r).
\end{equation}
At fixed $\delta$,
\begin{equation}
	r_\delta=-\rho^{-1}r^{-2},
	\qquad
	r_t=u.
\end{equation}
Thus $m=0$ and $\delta=M$ represent the center, while $m=M$ and $\delta=0$ represent the boundary.  Differentiation at fixed $m$ or fixed $\delta$ is material differentiation.

\subsection{Effective and transverse strain rate functions}

Define the effective velocity by
\begin{equation}
	\label{eq:def-v}
	v=u+\alpha\rho^{\alpha-2}\rho_r.
\end{equation}
It satisfies
\begin{equation}
	D_t v +F_1(\rho)(v-u) = 0.
\end{equation}
where
\begin{equation}
	\label{eq:F1-F2}
	F_1(\rho):=\frac{\gamma}{\alpha}\rho^{\gamma-\alpha},
	\qquad
	F_2(\rho):=
	\begin{cases}
		\displaystyle
		\frac{\gamma}{\alpha(\gamma-\alpha-1)}
		\rho^{\gamma-\alpha-1},&\gamma-\alpha\ne1,\\[6pt]
		\displaystyle \frac{\gamma}{\alpha}\log\rho,&\gamma-\alpha=1.
	\end{cases}
\end{equation}

The starting point for the following transverse strain rate functions we are about to introduce is the behavior at the center.  Smooth radial symmetry forces
\[
	u(t,0)=v(t,0)=0,
\]
so the values of the two velocities at the center contain no information about their first nontrivial behavior.  Their radial expansions are
\[
	u(t,r)=u_r(t,0)r+O(r^3),
	\qquad
	v(t,r)=v_r(t,0)r+O(r^3).
\]
Thus the relevant center quantities are the slopes $u_r(t,0)$ and $v_r(t,0)$. Since
\[
	u_r(t,0)=\lim_{r\to0}\frac{u(t,r)}r,
	\qquad
	v_r(t,0)=\lim_{r\to0}\frac{v(t,r)}r,
\]
we are led to extend these center slopes to the whole ball by treating
\[
	\frac ur,
	\qquad
	\frac vr
\]
as new unknowns.  These quotients are smooth at $r=0$ and retain the information that is lost in the identically vanishing center values of $u$ and $v$.

Define the transverse strain rate functions of physical and effective velocities by
\begin{equation}
	\U=\frac ur,
	\qquad
	\V=\frac vr.
\end{equation}
Their values at $r=0$ are understood through smooth radial limits as we just discussed. We also write
\begin{equation}
	\label{eq:angular-divergences}
	d:=\operatorname{div} u = u_r+2\frac ur,
	\qquad
	\omega:=\operatorname{div} v = v_r+2\frac vr,
\end{equation}
which is just an abbreviate name for the divergence of each velocity variable. We end this subsection by noting that these transverse strain rate functions satisfy the following equations, which themselves can be checked by direct calculation. 
\begin{align}
	\Dt\U
	&=\alpha\rho^{\alpha-1}
	\left(\U_{rr}+\frac{4}{r}\U_r\right)\notag\\[-2pt]
	&\quad+\left(\V-\U\right)
	\left[\alpha r\left(\U\right)_r
	+(3\alpha-2)\U-F_1(\rho)\right]-\U^2
	\label{eq:normalized-u-equation}\raisetag{0.5\baselineskip}\\[3pt]
	\Dt\V
	&=-\U\V-F_1(\rho)(\V-\U).
	\label{eq:angular-V}
\end{align}

\subsection{Dini maximum tools}

\begin{lemma}[Dini derivative of a spatial maximum]
	\label{lem:Dini}
	Let $K$ be a compact metric space.  Suppose $f:[0,t_1]\times K\to\mathbb R$ is continuous, $f_t$ exists, and $f_t$ is continuous.  Define
	\[
	M_f(t):=\max_{x\in K}f(t,x).
	\]
	Then
	\begin{equation}
		\label{eq:Dini-max}
		D^+M_f(t)
		\le\max_{x\in\operatorname{argmax}f(t,\cdot)}f_t(t,x),
	\end{equation}
	where $D^+$ denotes the upper right Dini derivative.
\end{lemma}

\begin{lemma}[Dini--Gronwall inequality]
	\label{lem:Dini-Gronwall}
	Let $g:[0,t_1]\to[0,\infty)$ be continuous.  If
	\[
	D^+g(t)\le Cg(t)
	\qquad(0\le t<t_1),
	\]
	then
	\[
	g(t)\le g(0)e^{Ct}
	\qquad(0\le t\le t_1).
	\]
\end{lemma}

\begin{proof}
	Set $h(t):=e^{-Ct}g(t)$.  The product rule for upper right Dini
	derivatives gives $D^+h(t)\le0$.  The Dini monotonicity criterion shows
	that $h$ is nonincreasing, and hence $g(t)\le g(0)e^{Ct}$.
\end{proof}

\subsection{Local theory and continuation under positive density}

We use the standard local theory under positive density in the following form.

For a positive radial pair $(\rho,\boldsymbol u)\in H^3(B_R)\times
(H^3(B_R)\cap H_0^1(B_R))$, define the equation acceleration
\begin{equation}
	\label{eq:equation-acceleration}
	\begin{aligned}
		\mathcal G_{\alpha,\gamma}(\rho,\boldsymbol u)
		:={}&\rho^{-1}\left[
		\divg(\rho^\alpha\nabla\boldsymbol u)
		+\nabla\bigl((\alpha-1)\rho^\alpha\divg\boldsymbol u\bigr)
		-\nabla\rho^\gamma
		\right]\\
		&-\boldsymbol u\mathbin{\cdot}\nabla\boldsymbol u.
	\end{aligned}
\end{equation}
For a radial vector field, $\nabla\boldsymbol u$ is symmetric and hence
$\mathbb D\boldsymbol u=\nabla\boldsymbol u$, so \eqref{eq:equation-acceleration} is the acceleration associated with \eqref{eq:ns-system} in the radial class.

\begin{proposition}[Compatible local existence and continuation under positive density]
	\label{prop:local-restart}
	Let $\alpha>1$ and $\gamma>1$.  Suppose $\rho_0$ and
	$\boldsymbol u_0$ are radial and
	\[
		0<c_0\le\rho_0\le C_0,
		\qquad
		\rho_0\in H^3(B_R),
		\qquad
		\boldsymbol u_0\in H^3(B_R)\cap H_0^1(B_R),
	\]
	and
	\[
		\mathcal G_{\alpha,\gamma}(\rho_0,\boldsymbol u_0)
		\in H_0^1(B_R).
	\]
	Then there is a time $T_0>0$, depending only on
	\[
		R,\alpha,\gamma,c_0^{-1},C_0,
		\norm{\rho_0}_{H^3},
		\norm{\boldsymbol u_0}_{H^3},
		\norm{\mathcal G_{\alpha,\gamma}(\rho_0,\boldsymbol u_0)}_{H^1},
	\]
	and a unique positive radial solution on $[0,T_0]$ satisfying
	\begin{align}
		\rho&\in C([0,T_0];H^3),
		&\rho_t&\in C([0,T_0];H^2),
	\label{eq:local-rho-class}\\
		\boldsymbol u&\in C([0,T_0];H^3\cap H_0^1)
		\cap L^2(0,T_0;H^4),
	\notag\\
		\boldsymbol u_t&\in L^\infty(0,T_0;H_0^1)
		\cap L^2(0,T_0;H^2),
		&\boldsymbol u_{tt}&\in L^2(0,T_0;L^2),
	\label{eq:local-u-class}
	\end{align}
	with $c_0/2\le\rho\le2C_0$.

	The same assertion applies at a translated initial time.  In particular, if a positive solution in the class \eqref{eq:local-rho-class}--\eqref{eq:local-u-class} has a strong terminal limit in $H^3\times H^3$ whose density remains positive, then the terminal pair automatically satisfies the same compatibility condition and the positive classical branch extends beyond that time.
\end{proposition}

\section{A priori estimates on bounded time intervals}
\label{sec:bounded-time}

This section records the estimates available in the exponent range
\eqref{eq:exponent-range}.  They have two roles: they turn the conditional
lower bound in the good boundary case into Corollary~\ref{cor:global-good-wall}, and they supply the upper bound for the density and endpoint criterion needed for the construction with negative boundary trace.

\subsection{Exponent choices and the initial data functional}

Let
\[
	\mathcal M_{\rm set}
	=\left\{1+\frac{m}{2j+1}:m\in\mathbb N^+,\ j\in\mathbb N_0\right\}.
\]
Fix exponents
\begin{equation}
	\label{eq:pq-range}
	q\in\mathcal M_{\rm set},
	\qquad
	\frac32<q<p,
	\qquad
	4p^2(\alpha-1)^2<(2p-1)\alpha(4\alpha-2),
\end{equation}
which is possible throughout \eqref{eq:exponent-range}; see
\cite[Definition~2.4 and Remark~2.11]{HuangMengZhang2025}.  Set
\begin{equation}
	\label{eq:s-def}
	s=\alpha-1+\frac1{2q}.
\end{equation}
For positive radial initial data define
\begin{equation}
	\label{eq:L0}
	\begin{aligned}
		\mathcal L_0:={}&
		\int_{B_R}\left(
		\rho_0+\rho_0^\gamma
		+\rho_0\abs{\boldsymbol u_0}^2
		+\rho_0\abs{\boldsymbol u_0}^{2p}
		+\abs{\nabla\rho_0^{\alpha-1/2}}^2
		\right)\dd x\\
		&+\int_{B_R}\left(
		\rho_0\abs{v_0}^2
		+\rho_0\abs{v_0}^{2q}
		+\abs{\nabla\rho_0^s}^{2q}
		\right)\dd x.
	\end{aligned}
\end{equation}

\subsection{Estimates on bounded time intervals and the criterion for density loss}

\begin{proposition}[Estimates on bounded time intervals and criterion for density loss]
	\label{prop:bounded-time}
	Assume \eqref{eq:exponent-range} and fix $p,q,s$ by
	\eqref{eq:pq-range}--\eqref{eq:s-def}.  Let
	$(\rho,\boldsymbol u)$ be the maximal positive radial classical solution
	issued from compatible positive $H^3$ data.  Then, for every finite
	$T_0>0$,
	\begin{equation}
		\label{eq:bounded-time-estimates}
		\sup_{0\le t<\min\{T^*,T_0\}}
		\left(
		\norm{\rho(t)}_{L^\infty(B_R)}
		+\norm{\nabla\rho^s(t)}_{L^{2q}(B_R)}
		+\int_{B_R}\rho\abs{\boldsymbol u}^{2p}\dd x
		\right)
		\le C(T_0,R,\mathcal L_0).
	\end{equation}
	If the maximal positive classical lifespan is finite, then
	\begin{equation}
		\label{eq:bounded-time-density-loss}
		\liminf_{t\uparrow T^*}
		\min_{\overline{B_R}}\rho(t)=0.
	\end{equation}
\end{proposition}

\begin{proof}
	Fix
	\[
		0<T<\min\{T^*,T_0\}.
	\]
	All constants below are uniform for such $T$.  We suppress the surface area factor $4\pi$ and use the outward mass coordinate from
	\eqref{eq:mass-coordinate}.  The identities recorded there give
	\[
		\partial_r=\rho r^2\partial_m,
		\qquad
		\int_0^MG(m)\dd m=\int_0^R\rho G(r)r^2\dd r.
	\]
	Time differentiation at fixed $m$ is the material derivative.

	We first derive the basic estimates directly from the radial equations.
	Using $d$ from \eqref{eq:angular-divergences}, the continuity equation and
	the radial momentum equation can be written as
	\[
		\Dt\rho=-\rho d
	\]
	and
	\[
		\rho\Dt u+(\rho^\gamma)_r
		-\bigl(\alpha\rho^\alpha d\bigr)_r
		+\frac2r(\rho^\alpha)_r u=0.
	\]
	
	\smallskip
	\noindent
	\emph{Step 1: Basic $L^2$ estimates.}
	Multiplying the radial momentum equation by $ur^2$ using continuity equation and integration by parts. We have
	\begin{equation}
			\begin{aligned}
			&\frac{\dd}{\dd t}\int_0^R
			\left(\frac12\rho u^2+\frac{\rho^\gamma}{\gamma-1}\right)r^2\dd r\\
			&\quad+\int_0^R\rho^\alpha\left(
			\alpha u_r^2r^2+4(\alpha-1)uu_rr+(4\alpha-2)u^2
			\right)\dd r=0.
		\end{aligned}
	\end{equation}
	Since $\alpha>1$, the dissipation is nonnegative and controls
	\[
		\int_0^R\rho^\alpha(u_r^2r^2+u^2)\dd r.
	\]

	We next derive the BD estimates. $v$ satisfies the equation
	\[
		\Dt v+\frac1\rho(\rho^\gamma)_r=0.
	\]
	Multiplication of the effective equation by $\rho vr^2$ gives
	\begin{equation}
		\begin{aligned}
			&\frac{\dd}{\dd t}\int_0^R
			\left(\frac12\rho v^2+\frac{\rho^\gamma}{\gamma-1}\right)r^2\dd r\\
			&\qquad+\alpha\gamma\int_0^R
			\rho^{\alpha+\gamma-3}\rho_r^2r^2\dd r=0.
		\end{aligned}
	\end{equation}

	The initial terms in both identities are bounded by $\mathcal L_0$.  Moreover,
	\[
		\abs{\partial_r\rho^{\alpha-1/2}}^2
		=\frac{(\alpha-1/2)^2}{\alpha^2}\rho\abs{v-u}^2
		\le C_\alpha\rho(v^2+u^2).
	\]
	The ordinary and BD energy identities therefore imply
	\begin{equation}
		\label{eq:hmz-base}
		\sup_{0\le t\le T}\left(
		\int_0^R\rho^\gamma r^2\dd r
		+\int_0^R
		\abs{\partial_r\rho^{\alpha-1/2}}^2r^2\dd r
		\right)\le C(R,\mathcal L_0),
	\end{equation}
	together with
	\[
	\begin{aligned}
		&\int_0^T\int_0^R\rho^\alpha(u_r^2r^2+u^2)\dd r\dd t\\
		&\qquad+\int_0^T\int_0^R
		\rho^{\alpha+\gamma-3}\rho_r^2r^2\dd r\dd t
		\le C(R,\mathcal L_0).
	\end{aligned}
	\]

	\smallskip
	\noindent
	\emph{Step 2: center weight estimate.}
	Let $f=\rho^{\alpha-1/2}$.  Since
	$\alpha-1/2<\gamma$, H\"older's inequality and
	\eqref{eq:hmz-base} give control
	\[
	\begin{aligned}
		\int_0^R\rho^{\alpha-1/2}r^2\,dr
		&=
		\int_0^R
		\left(\rho^\gamma r^2\right)^{\frac{\alpha-1/2}{\gamma}}
		\left(r^2\right)^{1-\frac{\alpha-1/2}{\gamma}}
		\,dr\\
		&\le
		\left(\int_0^R\rho^\gamma r^2\,dr\right)^{\frac{\alpha-1/2}{\gamma}}
		\left(\int_0^Rr^2\,dr\right)^{1-\frac{\alpha-1/2}{\gamma}}\\
		&\le C(R,\mathcal L_0).
	\end{aligned}
	\]
	Hence Sobolev--Poincar\'e on $B_R$ gives
	\[
		\|f\|_{L^6(B_R)}
		\le C_R\left(\|\nabla f\|_{L^2(B_R)}
		+\|f\|_{L^1(B_R)}\right)
		\le C(R,\mathcal L_0),
	\]
	or, in radial variables,
	\begin{equation}
		\label{eq:hmz-rho-power}
		\sup_{0\le t\le T}\int_0^R\rho^{6\alpha-3}r^2\dd r
		\le C(R,\mathcal L_0).
	\end{equation}
	We also retain the center weight.  Applying
	$W^{1,1}(0,R)\hookrightarrow L^\infty(0,R)$ to
	$fr^{1/2+\xi}$ yields
	\[
	\begin{aligned}
		\|fr^{1/2+\xi}\|_\infty
		\le C\bigg(&\int_0^Rfr^{1/2+\xi}\dd r
		+\int_0^R\abs{f_r}r^{1/2+\xi}\dd r\\
		&+\int_0^Rfr^{-1/2+\xi}\dd r\bigg).
	\end{aligned}
	\]
	The first term is controlled by \eqref{eq:hmz-rho-power}, 
	\[
	\begin{aligned}
		\int_0^Rfr^{1/2+\xi}\dd r
		&\le
		\left(\int_0^R(fr^{1/3})^6\,dr\right)^{1/6}
		\left(
		\int_0^R
		r^{(1/6+\xi)\frac65}\,dr
		\right)^{5/6}\\
		&=
		\left(\int_0^Rf^6r^2\,dr\right)^{1/6}
		\left(
		\int_0^Rr^{1/5+6\xi/5}\,dr
		\right)^{5/6}.
	\end{aligned}
	\]
	While H\"older's inequality gives
	\[
	\begin{aligned}
		\int_0^R\abs{f_r}r^{1/2+\xi}\dd r
		&\le\left(\int_0^R\abs{f_r}^2r^2\dd r\right)^{1/2}
		\left(\int_0^Rr^{-1+2\xi}\dd r\right)^{1/2},\\
		\int_0^Rfr^{-1/2+\xi}\dd r
		&\le\left(\int_0^Rf^6r^2\dd r\right)^{1/6}
		\left(\int_0^Rr^{-1+6\xi/5}\dd r\right)^{5/6}.
	\end{aligned}
	\]
	All radial weights are integrable for every $\xi>0$.  Thus
	\begin{equation}
		\label{eq:hmz-center-weight}
		\sup_{0\le t\le T}
		\|\rho^{\alpha-1/2}r^{1/2+\xi}\|_{L^\infty(0,R)}
		\le C_\xi(R,\mathcal L_0).
	\end{equation}

	\smallskip
	\noindent
	\emph{Step 3: $2p$ velocity moment estimate.}
	We next derive the $2p$ velocity moment from the radial momentum
	equation.  Multiplying it by
	$\abs{u}^{2p-2}ur^2$ and using the radial continuity equation gives
	\[
	\begin{aligned}
		&\frac1{2p}\frac{\dd}{\dd t}
		\int_0^R\rho\abs{u}^{2p}r^2\dd r\\
		&+\int_0^R\rho^\alpha\bigl[
		(4\alpha-2)\abs{u}^{2p}
		+(2p-1)\alpha\abs{u}^{2p-2}u_r^2r^2\\
		&\hspace{10em}
		+4p(\alpha-1)\abs{u}^{2p-2}u\,u_rr
		\bigr]\dd r\\
		&=\int_0^R\rho^\gamma\bigl[
		(2p-1)\abs{u}^{2p-2}u_rr^2
		+2\abs{u}^{2p-2}ur\bigr]\dd r.
	\end{aligned}
	\]
	The restriction on $p$ in \eqref{eq:pq-range} is precisely
	\[
		[4p(\alpha-1)]^2
		<4(4\alpha-2)(2p-1)\alpha,
	\]
	so the quadratic form on the left-hand side is positive	definite and dominate
	\[
	\rho^\alpha\left(
	|u|^{2p}
	+|u|^{2p-2}u_r^2r^2
	\right).
	\]
	Now we estimates those two terms on the right-hand side. With
	\[
		E:=2p(\gamma-\alpha)+\alpha,
	\]
	Young's inequality bounds the pressure terms by
	\[
	\begin{aligned}
		&\int_0^R\rho^\gamma\left(
		\abs{u}^{2p-2}\abs{u_r}r^2+\abs{u}^{2p-1}r\right)\dd r\\
		&\quad\le\varepsilon\int_0^R\rho^\alpha
		\left(\abs{u}^{2p}
		+\abs{u}^{2p-2}u_r^2r^2\right)\dd r
		+C_\varepsilon\int_0^R\rho^Er^{2p}\dd r.
	\end{aligned}
	\]

	For the density remainder. Let $a=\alpha-\frac12$. The center weight estimate gives
	\[
	\left\|\rho^a r^{1/2+\xi}\right\|_{L^\infty(0,R)}
	\le C_\xi(R,\mathcal L_0),
	\]
	for any $\xi>0$. Factor
	\[
	\begin{aligned}
		\rho^Er^{2p}
		&=
		\left(
		\rho^a r^{1/2+\xi}
		\right)^{E/a}
		r^{\,2p-\frac{E}{a}(1/2+\xi)}.
	\end{aligned}
	\]
	Therefore,
	\[
	\begin{aligned}
		\int_0^R\rho^Er^{2p}\,dr
		&\le
		\left\|\rho^ar^{1/2+\xi}\right\|_\infty^{E/a}
		\int_0^R
		r^{\,2p-\frac{E}{a}(1/2+\xi)}\,dr.
	\end{aligned}
	\]
	The radial integral is finite provided
	\[
	2p-\frac{E}{a}\left(\frac12+\xi\right)>-1.
	\]
	Equivalently,
	\[
	0<\xi<
	\frac{a(2p+1)}{E}-\frac12
	=
	\frac{
		p\left(3\alpha-1+\frac{\alpha-1}{2p}-\gamma\right)
	}{
		2p(\gamma-\alpha)+\alpha
	}.
	\]
	The upper endpoint is positive because
	$$
	\gamma<3\alpha-1
	<
	3\alpha-1+\frac{\alpha-1}{2p}.
	$$
	Thus, after choosing $\varepsilon$ small and integrating in time,
	we obtain
	\begin{align}
		{}&
		\sup_{0\le t\le T}\int_0^R\rho\abs{u}^{2p}r^2\dd r
		\notag\\
		&+\int_0^T\int_0^R\rho^\alpha
		\left(\abs{u}^{2p}
		+\abs{u}^{2p-2}u_r^2r^2\right)\dd r\dd t
		\le C(T_0,R,\mathcal L_0).
		\label{eq:hmz-Ap}
	\end{align}

	\smallskip
	\noindent
	\emph{Step 4: $2q$ effective velocity moment estimate.}
	For the effective velocity, multiply the mass-coordinate equation by $\abs{v}^{2q-2}v$. We have
	\[
	\frac1{2q}\frac{d}{dt}\int_0^M|v|^{2q}\,dm
	+
	\frac{\gamma}{\alpha}
	\int_0^M
	\rho^{\gamma-\alpha}
	|v|^{2q-2}v(v-u)\,dm
	=0.
	\]
	Since
	\[
		\abs{v-u}^{2q}
		\le C_q\left(
		\abs{v}^{2q-2}v(v-u)+\abs{u}^{2q}\right),
	\]
	we have
	\[
	\begin{aligned}
		&\frac1{2q}\frac{d}{dt}\int_0^M|v|^{2q}\,dm +
		c\int_0^M
		\rho^{\gamma-\alpha}|v-u|^{2q}\,dm\\
		&\le
		C\int_0^M
		\rho^{\gamma-\alpha}|u|^{2q}\,dm.
	\end{aligned}
	\]
	Integration over $(0,T)$ gives
	\begin{align}
		&\sup_{0\le t\le T}\int_0^M\abs{v}^{2q}\dd m
		+c\int_0^T\int_0^M
		\rho^{\gamma-\alpha}\abs{v-u}^{2q}\dd m\dd t
		\notag\\
		&\qquad\le C\mathcal L_0+C\int_0^T\int_0^R
		\rho^{\gamma-\alpha+1}\abs{u}^{2q}r^2\dd r\dd t.
		\label{eq:hmz-vq}
	\end{align}
	To estimate the last integral, set
	$\theta=(q-1)/(p-1)$.  Then
	\[
		\rho^\alpha\abs{u}^{2q}
		=(\rho^\alpha\abs{u}^{2p})^\theta
		(\rho^\alpha u^2)^{1-\theta},
	\]
	and consequently
	\[
	\begin{aligned}
		&\int_0^T\int_0^R
		\rho^{\gamma-\alpha+1}\abs{u}^{2q}r^2\dd r\dd t\\
		&\quad\le
		\sup_{[0,T]\times(0,R)}
		(\rho^{\gamma-2\alpha+1}r^2)
		\left(\int_0^T\int_0^R
		\rho^\alpha\abs{u}^{2p}\dd r\dd t\right)^\theta
		\left(\int_0^T\int_0^R
		\rho^\alpha u^2\dd r\dd t\right)^{1-\theta}.
	\end{aligned}
	\]
	The last two factors are controlled by \eqref{eq:hmz-Ap} and the
	ordinary energy dissipation.  If $k=\gamma-2\alpha+1$, then
	$0\le k<4\alpha-2$, and
	\[
		\rho^kr^2
		=\left(\rho^{\alpha-1/2}r^{1/2+\xi}\right)^{k/(\alpha-1/2)}
		r^{\,2-\frac{k}{\alpha-1/2}(1/2+\xi)}.
	\]
	For small $\xi>0$ the remaining power of $r$ is nonnegative; hence
	\eqref{eq:hmz-center-weight} controls the supremum above. Then we have
	\[
	\begin{aligned}
		&\sup_{0\le t\le T}\int_0^M|v(t,m)|^{2q}\,dm +
		c\int_0^T\int_0^M
		\rho^{\gamma-\alpha}|v-u|^{2q}\,dm\,dt\\
		&\le C(T_0,R,\mathcal L_0).
	\end{aligned}
	\]

	\smallskip
	\noindent
	\emph{Step 5: $L^\infty$ estimates of $\rho^s$.}
	By \eqref{eq:s-def},
	\[
		\partial_r\rho^s
		=\frac{s}{\alpha}\rho^{1/(2q)}(v-u),
		\qquad
		\int_0^R\abs{\partial_r\rho^s}^{2q}r^2\dd r
		=C_{\alpha,q}\int_0^M\abs{v-u}^{2q}\dd m.
	\]
	At each time, the triangle inequality and interpolation give
	\[
	\begin{aligned}
		\int_0^M\abs{v-u}^{2q}\dd m
		&\le C_q\int_0^M(\abs{v}^{2q}+\abs{u}^{2q})\dd m,\\
		\int_0^M\abs{u}^{2q}\dd m
		&\le\left(\int_0^M\abs{u}^{2p}\dd m\right)^{q/p}
		M^{1-q/p}.
	\end{aligned}
	\]
	Using \eqref{eq:hmz-Ap} and \eqref{eq:hmz-vq}, we conclude that
	\begin{equation}
		\label{eq:hmz-Dq}
		\sup_{0\le t\le T}
		\|\nabla\rho^s(t)\|_{L^{2q}(B_R)}
		\le C(T_0,R,\mathcal L_0).
	\end{equation}
	Finally, $s<2\alpha-1\le\gamma$, so the pressure energy controls
	$\|\rho^s\|_{L^1(B_R)}$.  Since $2q>3$, Sobolev--Poincar\'e and
	\eqref{eq:hmz-Dq} imply
	\[
		\|\rho^s(t)\|_{L^\infty(B_R)}
		\le C_R\left(
		\|\nabla\rho^s(t)\|_{L^{2q}(B_R)}
		+\|\rho^s(t)\|_{L^1(B_R)}\right)
		\le C(T_0,R,\mathcal L_0).
	\]
	Because $s>0$, this and \eqref{eq:hmz-Ap} prove
	\eqref{eq:bounded-time-estimates} after
	$T\uparrow\min\{T^*,T_0\}$.

	\smallskip
	\noindent
	\emph{Step 6: density loss criterion.}
	It remains to prove the density loss criterion.  Suppose that
	$T^*<\infty$ and that \eqref{eq:bounded-time-density-loss} fails.  Then,
	for some $t_0<T^*$,
	\[
		0<\underline\rho\le\rho(t,x)\le\overline\rho
		\qquad(t_0\le t<T^*,\ x\in B_R).
	\]
	Choose $3/2<\ell<\min\{q,p,3\}$.  Since powers of $\rho$ are now
	bounded above and below, \eqref{eq:hmz-base}, \eqref{eq:hmz-Ap},
	\eqref{eq:hmz-vq}, and \eqref{eq:hmz-Dq}, together with the ordinary
	and BD dissipations, yield
	\[
	\begin{aligned}
		&\sup_{t_0\le t<T^*}\left(
		\|\boldsymbol u\|_{L^{2\ell}}
		+\|\nabla\rho\|_{L^2}
		+\|\nabla\rho\|_{L^{2\ell}}\right)\\
		&\quad+\int_{t_0}^{T^*}\left(
		\|\nabla\rho\|_{L^2}^2
		+\|\nabla\rho\|_{L^{2\ell}}^{2\ell}
		+\|\nabla\boldsymbol u\|_{L^2}^2\right)\dd t\le C.
	\end{aligned}
	\]

	For completeness, we indicate the higher-order continuation argument.
	Returning to vector notation and writing
	$\dot{\boldsymbol u}=\boldsymbol u_t+
	\boldsymbol u\cdot\nabla\boldsymbol u$, testing the momentum equation
	by $\dot{\boldsymbol u}$ gives
	\[
	\begin{aligned}
		&\frac{\dd}{\dd t}\int_{B_R}\left(
		\rho^\alpha\abs{\nabla\boldsymbol u}^2
		+(\alpha-1)\rho^\alpha
		\abs{\divg\boldsymbol u}^2\right)\dd x\\
		&\qquad+\int_{B_R}\rho\abs{\dot{\boldsymbol u}}^2\dd x
		\le C\left(
		\|\nabla\rho\|_{L^2}^2
		+\|\nabla\boldsymbol u\|_{L^2}^2\right).
	\end{aligned}
	\]
	The coefficient bounds and elliptic regularity for the Dirichlet
	momentum system also give
	\[
		\|\boldsymbol u\|_{H^2}
		\le C\left(
		\|\sqrt\rho\,\dot{\boldsymbol u}\|_{L^2}
		+\|\nabla\rho\|_{L^2}
		+\|\nabla\boldsymbol u\|_{L^2}\right).
	\]
	After time integration,
	\[
		\sup_{t_0\le t<T^*}\|\nabla\boldsymbol u\|_{L^2}
		+\int_{t_0}^{T^*}
		\|\sqrt\rho\,\dot{\boldsymbol u}\|_{L^2}^2\dd t\le C.
	\]

	Next apply the material derivative to the momentum equation and test
	by $\dot{\boldsymbol u}$.  The commutators are controlled by the
	preceding lower-order and elliptic estimates, leading to
	\[
	\begin{aligned}
		&\frac{\dd}{\dd t}
		\|\sqrt\rho\,\dot{\boldsymbol u}\|_{L^2}^2
		+c\|\nabla\dot{\boldsymbol u}\|_{L^2}^2\\
		&\quad\le C\left[
		\|\sqrt\rho\,\dot{\boldsymbol u}\|_{L^2}^2
		+\|\sqrt\rho\,\dot{\boldsymbol u}\|_{L^2}^4
		+\|\nabla\rho\|_{L^2}^2
		+\|\nabla\rho\|_{L^{2\ell}}^{2\ell}
		+\|\nabla\boldsymbol u\|_{L^2}^2\right].
	\end{aligned}
	\]
	The coefficient of the quadratic term is time-integrable by the
	previous estimate, so Gronwall's inequality yields
	\[
		\sup_{t_0\le t<T^*}\left(
		\|\sqrt\rho\,\dot{\boldsymbol u}\|_{L^2}
		+\|\boldsymbol u\|_{H^2}\right)
		+\int_{t_0}^{T^*}
		\|\nabla\dot{\boldsymbol u}\|_{L^2}^2\dd t\le C.
	\]

	Differentiating the continuity equation twice and using the elliptic
	momentum equation now gives
	\[
		\sup_{t_0\le t<T^*}\|\rho\|_{H^2}
		+\int_{t_0}^{T^*}\|\boldsymbol u\|_{H^3}^2\dd t\le C.
	\]
	One time differentiation of the momentum equation, followed by three
	spatial derivatives of the continuity equation and two spatial
	derivatives of the elliptic momentum equation, closes the remaining
	estimates:
	\[
	\begin{aligned}
		&\sup_{t_0\le t<T^*}\left(
		\|\rho\|_{H^3}+\|\rho_t\|_{H^2}
		+\|\boldsymbol u\|_{H^3}
		+\|\boldsymbol u_t\|_{H^1}\right)\\
		&\qquad+\int_{t_0}^{T^*}\left(
		\|\boldsymbol u\|_{H^4}^2
		+\|\boldsymbol u_t\|_{H^2}^2
		+\|\boldsymbol u_{tt}\|_{L^2}^2\right)\dd t\le C.
	\end{aligned}
	\]
	These are the differentiated estimates of
	\cite[Propositions~3.15--3.19]{HuangMengZhang2025}.  Their coercivity
	here follows from
	\[
		\int_{B_R}\rho^\alpha\left(
		\abs{\nabla z}^2+(\alpha-1)\abs{\divg z}^2\right)\dd x
		\ge\underline\rho^\alpha\|\nabla z\|_{L^2}^2,
		\qquad z=\boldsymbol u,\dot{\boldsymbol u},
	\]
	which is favorable because $\alpha>1$.

	The Hilbert triple $H^4\subset H^3\subset H^2$ and the last bounds
	give a strong terminal limit for $\boldsymbol u$ in $H^3$.  For
	$t_0\le\tau<t<T^*$, the differentiated continuity equation gives
	\[
		\left|
		\|\rho(t)\|_{H^3}^2-\|\rho(\tau)\|_{H^3}^2
		\right|
		\le C\int_\tau^t
		\left(1+\|\boldsymbol u(\sigma)\|_{H^4}\right)\dd\sigma.
	\]
	Together with strong $H^2$ continuity and weak $H^3$ compactness, this
	shows that $\rho$ also has a strong terminal limit in $H^3$.  The limit
	remains between $\underline\rho$ and $\overline\rho$, so
	Proposition~\ref{prop:local-restart} restarts the positive classical
	solution beyond $T^*$.  This contradicts maximality and proves
	\eqref{eq:bounded-time-density-loss}.
\end{proof}

Proposition~\ref{prop:bounded-time} is the bridge between the two parts of
the paper.  In the good boundary case, its upper bound for the density activates Theorem~\ref{thm:main} and rules out a finite endpoint.  In the bad boundary case, the same estimates persist uniformly along the constructed family, while the criterion for density loss identifies what must fail at the finite endpoint.

\section{Nonnegative boundary trace: positive density preservation}

The proof of Theorem~\ref{thm:main} has three analytic stages.  We first
control the positive parts of the transverse strain rate functions of physical and effective velocities.  We then propagate a pathwise upper bound for an effective pressure variable.  Finally, the boundary sign supplies the favorable boundary condition for a maximum principle applied to $1/\rho$.

\subsection{Mass normalization and maximum point equations}

The first stage of the proof is a coupled maximum principle estimate for the
physical and effective velocities. To formulate that argument on the closed
mass interval, we renormalize the velocities by the center scale
$m^{1/3}$. Since $m^{1/3}$ vanishes linearly with $r$ at the center, the
resulting variables retain the first nontrivial center behavior of $u$ and
$v$ and extend smoothly to $m=0$. We now introduce these variables and record
the maximum point identities used in the next subsection. Throughout this
subsection, time derivatives are taken at fixed outward mass label.

Define
\begin{equation}
	\label{eq:sigma-P-Z}
	\sigma(t,m):=\frac{r(t,m)}{m^{1/3}},
	\qquad
	\Pcal:=\sigma\U=\frac{u}{m^{1/3}},
	\qquad
	\Z:=\sigma\V=\frac{v}{m^{1/3}},
\end{equation}
using smooth center limits at $m=0$. Then at fixed outward mass label, 
\begin{equation}
	\label{eq:sigma-Z-evolution}
	\sigma_t=\Pcal,
	\qquad
	\Z_t=F_1(\rho)(\Pcal-\Z).
\end{equation}
Substitution of the equation for the transverse strain rate function of  velocity \eqref{eq:normalized-u-equation} gives the following two record identities.
Let $e:=\frac{\rho r^3}{3m}$.  At every spatial maximum of $\Pcal$ attained at
positive radius,
\begin{equation}
	\label{eq:P-record}
	\Pcal_t
	=\alpha\rho^{\alpha-1}\Pcal_{rr}
	-\frac{2\alpha\rho^{\alpha-1}(1-e)^2}{r^2}\Pcal
	+(\Pcal-\Z)\left(F_1-C_e\U\right),
\end{equation}
where
\begin{equation}
	\label{eq:Ce}
	C_e:=2(\alpha-1)+(\alpha+1)e\ge2(\alpha-1).
\end{equation}
At a positive maximum of $\Pcal$ at the center one has the separate exact identity
\begin{equation}
	\label{eq:P-record-center}
	\Pcal_t(t,0)
	=5\alpha\rho(t,0)^{\alpha-1}\Pcal_{rr}(t,0)
	+(\Pcal(t,0)-\Z(t,0))
	\left(F_1(\rho(t,0))-(3\alpha-1)\U(t,0)\right).
\end{equation}

The positive-radius and center formulas must be kept separate: at the center,
the radial operator $\partial_{rr}+4r^{-1}\partial_r$ contributes
$5\partial_{rr}$, which produces \eqref{eq:P-record-center} rather than the
formal $r\downarrow0$ limit of \eqref{eq:P-record}. Together with
\eqref{eq:sigma-Z-evolution}, these identities provide the sign mechanism for
the joint positive maximum of $\Pcal$ and $\Z$, to which we now turn.

\subsection{Maximum principle bounds for transverse strain rate functions}
Here and after for any real-valued function $f$, its positive and negative parts are defined as
\[
f_+:=\max\{f,0\},
\qquad
f_-:=\max\{-f,0\}.
\]
\begin{proposition} \label{prop:angular-bounds}
	There exists a constant $C_T$ such that
	\begin{equation}
		\label{eq:angular-positive-bounds}
		\sup_{0\le t<T,\,0\le r\le R}
		\U_+(t,r)
		+\sup_{0\le t<T,\,0\le r\le R}
		\V_+(t,r)
		\le C_T,
	\end{equation}
	and, along every fixed mass label $m\in[0,M]$,
	\begin{equation}
		\label{eq:path-angular-u-minus}
		\int_0^t
		\U_-(s,r(s,m))\dd s\le C_T
		\qquad(0\le t<T).
	\end{equation}
\end{proposition}

\begin{proof}
	We use an argument based on the maximum principle.  Set
	\[
	\mathcal S(t):=\max_{0\le m\le M}\sigma(t,m),
	\qquad
	\mathcal N(t):=\max_{0\le m\le M}\max\{\Pcal_+(t,m),\Z_+(t,m)\}.
	\]
	The upper density bound gives
	\[
	0\le F_1(\rho)\le F_T:=\frac{\gamma}{\alpha}\overline\rho_T^{\gamma-\alpha}.
	\] 
	Define $C_A:=\frac{F_T}{2(\alpha-1)}$. Fix a time at which
	\[
	\mathcal N>C_A\mathcal S.
	\]
	If the maximum is attained by $\Pcal$, then at the maximizing label
	$\Pcal \ge \Z$ and
	\[
	\U=\frac{\Pcal}{\sigma}
	>C_A\frac{\mathcal S}{\sigma}
	\ge C_A\ge\frac{F_1}{C_e}.
	\]
	The boundary cannot be this positive maximizer because
	$\Pcal(t,M)=u(t,R)/M^{1/3}=0$.  Therefore every term on the right-hand side
	of \eqref{eq:P-record} is nonpositive.  If the maximizing label is the
	center, use \eqref{eq:P-record-center} instead: there
	$\Pcal_{rr}\le0$, $3\alpha-1\ge2(\alpha-1)$, and the same sign argument
	again makes the right-hand side nonpositive.  Thus
	\[
	\Pcal_t\le0.
	\]
	If the active maximum is instead attained by $\Z$, then
	$\Z\ge \Pcal$ and \eqref{eq:sigma-Z-evolution} gives
	\[
	\Z_t\le0.
	\]
	
	Define
	\[
	\mathcal S^*(t):=\max_{0\le s\le t}\mathcal S(s),
	\qquad
	\mathcal N^*(t):=\max_{0\le s\le t}\mathcal N(s),
	\qquad
	N_0:=\mathcal N(0).
	\]
	For $\varepsilon>0$, set
	\[
	\mathcal N_\varepsilon(t)
	:=\max\left\{
	0,
	\max_{0\le m\le M}(\Pcal(t,m)-\varepsilon t),
	\max_{0\le m\le M}(\Z(t,m)-\varepsilon t)
	\right\}.
	\]
	Whenever
	\begin{equation}
		\label{eq:contradiction-condition}
		\mathcal N_\varepsilon(t)>0,
		\qquad
		\mathcal N_\varepsilon(t)+\varepsilon t>C_A\mathcal S(t),
	\end{equation}
	the preceding pointwise calculation gives
	\[
	\partial_t(\Pcal-\varepsilon t)\le-\varepsilon
	\quad\hbox{or}\quad
	\partial_t(\Z-\varepsilon t)\le-\varepsilon,
	\]
	according to which component is larger.  Lemma~\ref{lem:Dini}, applied on
	the compact disjoint union consisting of the constant zero component and two
	copies of $m\in[0,M]$, therefore gives
	\begin{equation}
		\label{eq:Nepsilon-Dini}
		D^+\mathcal N_\varepsilon(t)\le-\varepsilon.
	\end{equation}
	The continuous function
	\[
	H(t):=\max\{N_0,C_A\mathcal S^*(t)\}
	\]
	is nondecreasing and $\mathcal N_\varepsilon(0)\le H(0)$.  If
	$\mathcal N_\varepsilon$ crossed $H$ from below for the first time at some
	$t_0>0$, condition \eqref{eq:contradiction-condition} would hold there.
	Inequality \eqref{eq:Nepsilon-Dini} is incompatible with such a first
	crossing.  Thus
	\[
	\mathcal N_\varepsilon(t)\le H(t)
	\qquad(0\le t<T).
	\]
	Letting $\varepsilon\downarrow0$ gives
	$\mathcal N(t)\le H(t)$, and taking the time supremum gives
	\begin{equation}
		\label{eq:Nstar}
		\mathcal N^*(t)\le\max\{N_0,C_A\mathcal S^*(t)\}.
	\end{equation}
	
	Since $\sigma_t=\Pcal\le\mathcal N$,
	\eqref{eq:Nstar} implies
	\[
	\mathcal S^*(t)
	\le S_0+N_0t+C_A\int_0^t\mathcal S^*(s)\dd s,
	\qquad S_0:=\mathcal S(0).
	\]
	Gronwall's inequality yields
	\begin{equation}
		\label{eq:Sstar}
		\mathcal S^*(t)\le(S_0+N_0t)e^{C_At}.
	\end{equation}
	Thus $\mathcal N$ is uniformly bounded on $[0,T)$.
	
	Only the upper density bound is needed for a lower bound on $\sigma$:
	\[
	m(t,r)=\int_0^r\rho(t,s)s^2\dd s
	\le\frac{\overline\rho_T}{3}r^3.
	\]
	Therefore
	\begin{equation}
		\label{eq:sigma-lower}
		\sigma(t,m)\ge\sigma_*:=
		\left(\frac3{\overline\rho_T}\right)^{1/3}>0.
	\end{equation}
	Equations \eqref{eq:Nstar}--\eqref{eq:sigma-lower} give uniform bounds for $\Pcal_+/\sigma=(u/r)_+$ and $\Z_+/\sigma=(v/r)_+$, proving
	\eqref{eq:angular-positive-bounds}.
	
	Finally, $\sigma_t=\sigma \U$ gives, on every fixed mass trajectory,
	including the smooth center limit,
	\begin{equation}
		\label{eq:integral-angular-u}
		\int_0^t\U(s,r(s,m))\dd s
		=\log\frac{\sigma(t,m)}{\sigma(0,m)}.
	\end{equation}
	The lower bound \eqref{eq:sigma-lower} and upper bound
	\eqref{eq:Sstar} bound the right-hand side from both sides uniformly.  Since
	\[
	\U
	=\U_+ -\U_-,
	\qquad
	\U_+\le C_T.
	\]
	Therefore \eqref{eq:integral-angular-u} implies
	\[
	\int_0^t\U_-\dd s
	=\int_0^t\U_+\dd s
	-\int_0^t\U\dd s\le C_T,
	\]
	which proves \eqref{eq:path-angular-u-minus} without using a density lower bound.
\end{proof}

\subsection{Effective pressure}
Define the effective pressure
\begin{equation}
	\label{eq:def-Y}
	\Y:=\frac{v_r}{\rho}+F_2(\rho).
\end{equation}
Differentiating the equation for the effective velocity and using the continuity
equation gives
\begin{equation}
	\label{eq:Y-equation}
	\Dt\Y+\left(F_1-2\U\right)\Y
	=F_1F_2-2\U\left(\frac{F_1}{\rho}+F_2\right)
	-\frac{F_1'(\rho)}{\mu'(\rho)}(v-u)^2.
\end{equation}
Set
\begin{equation}
	\label{eq:def-Q}
	\Q(t):=\max_{0\le r\le R}\frac1{\rho(t,r)}.
\end{equation}

\begin{proposition}[Pathwise upper bounds for $\Y$]
	\label{prop:Y-bounds}
	There exists a constant $C_T$, depending only on the parameters, the
	upper bound for the density on the finite time interval, and the initial
	data, such that, along every fixed mass label $m\in[0,M]$, the following
	estimates hold for all $t<T$:
	\begin{align}
		\Y(t,m)
		&\le C_T+C_T\int_0^t
		\rho(s,m)^{\gamma-\alpha-1}\dd s,
		&&0<\gamma-\alpha<1,
		\label{eq:Y-path-sublinear}\\
		\Y(t,m)
		&\le C_T+C_T\int_0^t
		\log\left(1+\frac1{\rho(s,m)}\right)\dd s,
		&&\gamma-\alpha=1,
		\label{eq:Y-path-log}\\
		\Y(t,m)&\le C_T,
		&&\gamma-\alpha>1.
		\label{eq:Y-path-uniform}
	\end{align}
	The constants are uniform in $m$.  Consequently, with $\Q$ defined in
	\eqref{eq:def-Q},
	\begin{align}
		\sup_{0\le r\le R}\Y(t,r)
		&\le C_T+C_T\int_0^t
		\Q(s)^{1-(\gamma-\alpha)}\dd s,
		&&0<\gamma-\alpha<1,
		\label{eq:Y-bound-sublinear}\\
		\sup_{0\le r\le R}\Y(t,r)
		&\le C_T+C_T\int_0^t\log(1+\Q(s))\dd s,
		&&\gamma-\alpha=1,
		\label{eq:Y-bound-log}\\
		\sup_{0\le t<T,\,0\le r\le R}\Y(t,r)
		&\le C_T,
		&&\gamma-\alpha>1.
		\label{eq:Y-bound-uniform}
	\end{align}
	No lower bound for $\rho$ at later times is used in any of these estimates.
\end{proposition}

\begin{proof}
	Along a fixed mass label, the integrating kernel associated with
	\eqref{eq:Y-equation} is
	\[
	K(s,t;m):=
	\exp\left(
	-\int_s^t\left(F_1-2\U\right)
	(\tau,r(\tau,m))\dd\tau
	\right).
	\]
	For $m=0$ and $m=M$, the same formula is understood through the smooth
	center and boundary traces.  Since $\U_+\le C_T$ by
	Proposition~\ref{prop:angular-bounds} and $F_1\ge0$, we have
	\begin{equation}
		\label{eq:kernel-bound}
		0<K(s,t;m)\le e^{2C_TT}.
	\end{equation}
	The corresponding exact representation is
	\begin{equation}
		\label{eq:Y-representation}
		\begin{aligned}
			\Y(t,m)
			&=K(0,t;m)\Y(0,m)\\
			&\quad+\int_0^tK(s,t;m)
			\left[
			F_1F_2-2\U \left(\frac{F_1}{\rho}+F_2\right)-\frac{F_1'}{\mu'}(v-u)^2
			\right](s,m)\dd s.
		\end{aligned}
	\end{equation}
	The last term in \eqref{eq:Y-equation} is nonpositive because
	\begin{equation}
		\label{eq:Fprime-muprime}
		\frac{F_1'(\rho)}{\mu'(\rho)}
		=\frac{\gamma(\gamma-\alpha)}{\alpha^2}
		\rho^{(\gamma-\alpha)-\alpha}\ge0.
	\end{equation}
	Define
	\begin{equation}
		\label{eq:def-A}
		A(\rho):=\frac{F_1(\rho)}{\rho}+F_2(\rho).
	\end{equation}
	If $\gamma-\alpha\ne1$, then
	\begin{equation}
		\label{eq:A-power}
		A(\rho)
		=\frac{\gamma(\gamma-\alpha)}
		{\alpha(\gamma-\alpha-1)}
		\rho^{\gamma-\alpha-1},
	\end{equation}
	whereas for $\gamma-\alpha=1$,
	\begin{equation}
		\label{eq:A-log}
		A(\rho)=\frac{\gamma}{\alpha}(1+\log\rho).
	\end{equation}
	
	Suppose first that $0<\gamma-\alpha<1$.  Then $F_1F_2\le0$ and $A<0$.
	Consequently,
	\[
	\pos{-2\U A}
	\le2\U_+|A|
	\le C_T\rho^{\gamma-\alpha-1}.
	\]
	The other two source terms in \eqref{eq:Y-representation} are
	nonpositive.  Hence \eqref{eq:Y-representation} and
	\eqref{eq:kernel-bound} give, along every fixed mass label $m$,
	\[
	\begin{aligned}
		\Y(t,m)
		&\le K(0,t;m)\Y_+(0,m)
		+C_T\int_0^t
		K(s,t;m)\rho(s,m)^{\gamma-\alpha-1}\dd s\\
		&\le C_T+C_T\int_0^t
		\rho(s,m)^{\gamma-\alpha-1}\dd s.
	\end{aligned}
	\]
	Here the last constant absorbs the uniformly bounded initial contribution
	and the kernel bound.  This proves \eqref{eq:Y-path-sublinear}.  Since
	\[
	\rho(s,m)^{\gamma-\alpha-1}
	=\rho(s,m)^{-(1-(\gamma-\alpha))}
	\le\Q(s)^{1-(\gamma-\alpha)},
	\]
	taking the supremum over $m\in[0,M]$, equivalently over $r\in[0,R]$,
	proves \eqref{eq:Y-bound-sublinear}.
	
	If $\gamma-\alpha=1$, the positive part of
	\[
	F_1F_2=\frac{\gamma^2}{\alpha^2}\rho\log\rho
	\]
	is bounded in terms of $\overline\rho_T$.  Equations \eqref{eq:upper-density} and \eqref{eq:A-log} also give
	\[
	A_+(\rho)\le C_{\overline\rho_T},
	\qquad
	A_-(\rho)\le C\log(1+\rho^{-1}).
	\]
	Thus
	\[
	\pos{-2\U A}
	\le2\U_+A_-
	+2\U_-A_+
	\le C_T\log(1+\rho^{-1})
	+C\U_-.
	\]
	Since the square term in \eqref{eq:Y-representation} is nonpositive,
	\eqref{eq:Y-representation} and \eqref{eq:kernel-bound} therefore give,
	along every fixed mass label $m$,
	\[
	\begin{aligned}
		\Y(t,m)
		&\le C_T
		+C_T\int_0^t
		\left[1+\log\left(1+\frac1{\rho(s,m)}\right)
		+\U_-(s,r(s,m))\right]\dd s.
	\end{aligned}
	\]
	Here the constant term absorbs the uniformly bounded initial contribution.
	Because $t<T$ and \eqref{eq:path-angular-u-minus} gives
	\[
	\int_0^t\U_-(s,r(s,m))\dd s\le C_T
	\]
	uniformly in $m$, we conclude that
	\[
	\Y(t,m)
	\le C_T+C_T\int_0^t
	\log\left(1+\frac1{\rho(s,m)}\right)\dd s.
	\]
	This proves \eqref{eq:Y-path-log}.  Since
	$\rho(s,m)^{-1}\le\Q(s)$,
	taking the supremum over $m\in[0,M]$, equivalently over
	$r\in[0,R]$, proves \eqref{eq:Y-bound-log}.
	
	Finally, if $\gamma-\alpha>1$, both $F_1F_2$ and $A$ are nonnegative and
	uniformly bounded above by \eqref{eq:upper-density}.  Hence
	\[
	\pos{-2\U A}
	\le C\left(\U\right)_-.
	\]
	Another use of \eqref{eq:path-angular-u-minus} proves
	\eqref{eq:Y-path-uniform}, uniformly in $m$.  Taking the supremum proves
	\eqref{eq:Y-bound-uniform}.
	
	In all three cases, the initial contribution is finite because the smooth
	initial density is strictly positive and
	\[
	\max_{0\le r\le R}
	\left(\frac{(v_0)_r}{\rho_0}+F_2(\rho_0)\right)_+<\infty.
	\]
	This completes the proof.
\end{proof}

\subsection{The maximum principle for reciprocal density}
Let
\[
	z:=\frac1\rho.
\]
Then
\begin{equation}
	\label{eq:reciprocal-equation}
	\Dt z
	=z\omega
	+\alpha z^{1-\alpha}
	\left(z_{rr}+\frac2rz_r\right)
	-\alpha^2z^{-\alpha}z_r^2.
\end{equation}
By \eqref{eq:def-Y} and \eqref{eq:angular-divergences},
\[
	z\omega
	=\frac{v_r}{\rho}+2\frac{v}{r\rho}
	=\Y-F_2(\rho)+2\V z.
\]
Therefore \eqref{eq:reciprocal-equation} is equivalently
\begin{equation}
	\label{eq:reciprocal-Y-equation}
	\Dt z
	=\Y-F_2(\rho)+2\V z
	+\alpha z^{1-\alpha}
	\left(z_{rr}+\frac2rz_r\right)
	-\alpha^2z^{-\alpha}z_r^2.
\end{equation}

We also record the corresponding equation in the outward mass coordinate.
At fixed $m$, time differentiation is material differentiation, and
\eqref{eq:mass-coordinate} gives
\begin{equation}
	\label{eq:reciprocal-mass-derivatives}
	m_r=\rho r^2=\frac{r^2}{z},
	\qquad
	\partial_r=\frac{r^2}{z}\partial_m,
	\qquad
	r_m=\frac{z}{r^2}.
\end{equation}
By direct calculation, we also have
\begin{equation}
	\label{eq:reciprocal-mass-equation}
	\begin{aligned}
	\left.z_t\right|_m
	={}&\Y-F_2(\rho)+2\V z
	+\alpha r^4z^{-\alpha-1}z_{mm}
	+4\alpha rz^{-\alpha}z_m\\
	&-\alpha(\alpha+1)r^4z^{-\alpha-2}z_m^2.
	\end{aligned}
\end{equation}
Writing $V(t):=v(t,R)$, at the boundary,
\begin{equation}
	\label{eq:wall-zr}
	z_r(t,R)=-\frac{V(t)}{\alpha}z(t,R)^\alpha.
\end{equation}

\begin{proposition}[Maximum inequality for nonnegative boundary trace]
	\label{prop:Q-inequality}
	Under the hypotheses of Theorem~\ref{thm:main},
	\begin{equation}
		\label{eq:Q-master}
		D^+\Q(t)
		\le
		\sup_{0\le r\le R}\Y(t,r)
		-F_2(\Q(t)^{-1})
		+2C_T\Q(t).
	\end{equation}
\end{proposition}

\begin{proof}
	Fix $t<T$ and consider a maximizer of $z(t,\cdot)$.  We first verify that
	every possible maximizing location has the favorable diffusion sign.
	
	If $V(0)>0$, the boundary formula
	\eqref{eq:intro-wall-ode} and
	\eqref{eq:wall-zr} give $z_r(t,R)<0$.  A differentiable function attaining a
	maximum at the right endpoint must have left derivative at least zero.
	Therefore the boundary cannot maximize $z$ in the strict good boundary case.
	
	If $V(0)=0$, then $V(t)=0$ and $z_r(t,R)=0$.  If the boundary maximizes $z$,
	Taylor expansion from the left gives
	\[
	z(t,R-\delta)
	=z(t,R)+\frac12z_{rr}(t,R)\delta^2+o(\delta^2)
	\le z(t,R),
	\]
	so $z_{rr}(t,R)\le0$.  Moreover $u(t,R)=v(t,R)=0$, hence
	\[
	\frac{u(t,R)}{R}=\frac{v(t,R)}{R}=0.
	\]
	
	At an interior maximizer, $z_r=0$ and $z_{rr}\le0$.  At a center maximizer,
	smooth radial parity gives $z_r(t,0)=0$, $z_{rr}(t,0)\le0$, and
	\[
	\lim_{r\downarrow0}
	\left(z_{rr}+\frac2rz_r\right)
	=3z_{rr}(t,0)\le0.
	\]
	Thus the last two terms of \eqref{eq:reciprocal-equation} are nonpositive at
	every admissible maximizer.  Also $u z_r=0$ there: this follows from $z_r=0$
	at an interior or center maximum and from $u=0$ at the boundary.  Hence
	\[
	z_t\le z\omega=\frac{\omega}{\rho}
	\]
	at every maximizer.  Define the set of current maximizers by
	\[
	\mathcal A(t):=
	\left\{r\in[0,R]:z(t,r)=\Q(t)\right\}.
	\]
	Applying Lemma~\ref{lem:Dini} and then the preceding pointwise estimate gives
	\[
	D^+\Q(t)
	\le \max_{r\in\mathcal A(t)}z_t(t,r)
	\le \max_{r\in\mathcal A(t)}\frac{\omega(t,r)}{\rho(t,r)}.
	\]
	
	Finally, by \eqref{eq:def-Y} and \eqref{eq:angular-divergences},
	\[
	\frac{\omega}{\rho}
	=\Y-F_2(\rho)+2\frac{v}{r}z.
	\]
	At every maximizer, $z=\Q$ and $\rho=\Q^{-1}$.  Since
	\[
	\frac{v}{r}\le\left(\frac{v}{r}\right)_+\le C_T
	\]
	by Proposition~\ref{prop:angular-bounds}.  Consequently, for every
	$r\in\mathcal A(t)$,
	\[
	\begin{aligned}
		\frac{\omega(t,r)}{\rho(t,r)}
		&=\Y(t,r)-F_2\left(\Q(t)^{-1}\right)
		+2\frac{v(t,r)}{r}\Q(t)\\
		&\le \Y(t,r)-F_2\left(\Q(t)^{-1}\right)
		+2C_T\Q(t)\\
		&\le \sup_{0\le s\le R}\Y(t,s)
		-F_2\left(\Q(t)^{-1}\right)+2C_T\Q(t).
	\end{aligned}
	\]
	Taking the maximum over $r\in\mathcal A(t)$ in the preceding Dini estimate
	therefore gives \eqref{eq:Q-master}.
\end{proof}

\subsection{Proof of the density lower bound}
\begin{proof}[Proof of Theorem~\ref{thm:main}]
	We close \eqref{eq:Q-master} separately in the three possible regimes for
	$\gamma-\alpha$.
	
	\smallskip
	\noindent\emph{Case 1: $0<\gamma-\alpha<1$.}
	Here
	\[
	-F_2(\Q^{-1})
	=\frac{\gamma}{\alpha(1-(\gamma-\alpha))}
	\Q^{1-(\gamma-\alpha)}.
	\]
	Combining Proposition~\ref{prop:Q-inequality} with
	\eqref{eq:Y-bound-sublinear} gives
	\begin{equation}
		\label{eq:Q-sublinear}
		D^+\Q(t)
		\le C_T(1+\Q(t))
		+C_T\Q(t)^{1-(\gamma-\alpha)}
		+C_T\int_0^t\Q(s)^{1-(\gamma-\alpha)}\dd s.
	\end{equation}
	Since $0<1-(\gamma-\alpha)<1$,
	$x^{1-(\gamma-\alpha)}\le1+x$ for $x\ge0$.  Applying this
	inequality both to the pointwise term and inside the time integral in
	\eqref{eq:Q-sublinear}, and then increasing $C_T$, gives
	\[
	\begin{aligned}
		D^+\Q(t)
		&\le C_T(1+\Q(t))
		+C_T\int_0^t(1+\Q(s))\dd s.
	\end{aligned}
	\]
	Define
	\[
	G(t):=1+\Q(t)+\int_0^t(1+\Q(s))\dd s.
	\]
	Thus $D^+\Q(t)\le C_TG(t)$.  Since $\Q$ is continuous, the integral
	in the definition of $G$ is differentiable.  Adding a differentiable
	function to an upper right Dini derivative therefore yields
	\[
	\begin{aligned}
		D^+G(t)
		&=D^+\Q(t)+1+\Q(t)\\
		&\le C_TG(t)+1+\Q(t)\\
		&\le (C_T+1)G(t).
	\end{aligned}
	\]
	After increasing $C_T$ once more, this is
	\[
	D^+G(t)\le C_TG(t).
	\]
	Lemma~\ref{lem:Dini-Gronwall} yields
	$G(t)\le G(0)e^{C_Tt}$ for every $t<T$.
	
	\smallskip
	\noindent\emph{Case 2: $\gamma-\alpha=1$.}
	At a maximum of the reciprocal density,
	\[
	-F_2(\Q^{-1})=\frac{\gamma}{\alpha}\log\Q
	\le C\log(1+\Q).
	\]
	Using \eqref{eq:Y-bound-log} in \eqref{eq:Q-master}, followed by
	$\log(1+x)\le x$, gives the same inequality
	\[
	D^+\Q(t)
	\le C_T\left(1+\Q(t)
	+\int_0^t(1+\Q(s))\dd s\right).
	\]
	The preceding function $G$ is therefore bounded by
	Lemma~\ref{lem:Dini-Gronwall}.
	
	\smallskip
	\noindent\emph{Case 3: $\gamma-\alpha>1$.}
	In this range $F_2>0$, so the $-F_2$ term in \eqref{eq:Q-master} is
	nonpositive.  Estimate \eqref{eq:Y-bound-uniform} gives
	\[
	D^+\Q(t)\le C_T(1+\Q(t)).
	\]
	Applying Lemma~\ref{lem:Dini-Gronwall} to $1+\Q$ again bounds $\Q$
	uniformly on $[0,T)$.
	
	In every case,
	\[
	\sup_{0\le t<T}\Q(t)\le C_T'<\infty.
	\]
	Since $\Q(t)=\|\rho(t,\cdot)^{-1}\|_{L^\infty(0,R)}$, setting
	$c_T=(C_T')^{-1}$ proves \eqref{eq:main-conclusion}.
\end{proof}

\begin{proof}[Proof of Corollary~\ref{cor:global-good-wall}]
	Let $[0,T^*)$ be the maximal lifespan of the positive classical solution.  Suppose for
	contradiction that $T^*<\infty$.  Proposition~\ref{prop:bounded-time},
	applied with any fixed $T_0>T^*$, gives
	\[
	\sup_{0\le t<T^*}\norm{\rho(t)}_{L^\infty(B_R)}<\infty.
	\]
	Theorem~\ref{thm:main} and $v_0(R)\ge0$ therefore imply
	\[
	\inf_{0\le t<T^*}\min_{\overline{B_R}}\rho(t)>0.
	\]
	This contradicts the criterion for density loss
	\eqref{eq:bounded-time-density-loss}.  Hence $T^*=\infty$.

	For an arbitrary finite $T_0>0$, apply
	Proposition~\ref{prop:bounded-time} with the slightly larger time bound
	$T_0+1$ to obtain the upper bound through time $T_0$.  Theorem~\ref{thm:main}
	then gives the corresponding lower bound.  This proves the asserted
	two-sided estimate on every bounded time interval.
\end{proof}

\subsection{Interior vacuum localization}

Although the preceding full-ball lower bound uses the nonnegative boundary trace, the angular estimates in Proposition~\ref{prop:angular-bounds} and the pathwise effective-pressure estimates in Proposition~\ref{prop:Y-bounds} do not use that sign.  We now combine those estimates with an inward-mass weight which forces the reciprocal-density variable to vanish at the boundary.

Recall the inward mass distance $\delta=M-m$ and set
\begin{equation}
	\label{eq:def-weighted-reciprocal}
	\phi(m):=\delta^{1/\alpha},
	\qquad
	X(t,m):=\phi(m)z(t,m)
	=\frac{(M-m)^{1/\alpha}}{\rho(t,m)}.
\end{equation}
Because $m$ is a material coordinate, both $\delta$ and $\phi$ are independent
of time along a fixed mass label.  Moreover,
\begin{equation}
	\label{eq:weighted-reciprocal-wall}
	X(t,M)=0
\end{equation}
for every $t<T$: the density is positive at each preterminal time, while
$\phi(M)=0$.  Thus a positive maximum of $X$ cannot occur at the boundary,
regardless of the sign of $v_0(R)$.

\begin{proposition}[Critical weighted maximum inequality]
	\label{prop:weighted-reciprocal-maximum}
	At every positive interior maximum of $X$, and also at a maximum at the
	center, one has
	\begin{equation}
		\label{eq:weighted-reciprocal-maximum}
		\left.X_t\right|_m
		\le
		\phi\Y-\phi F_2(\rho)+2\V X+4rX^{1-\alpha}.
	\end{equation}
\end{proposition}

\begin{proof}
	Consider first an interior mass label $0<m<M$ at which $X$ is maximal.
	Since
	\[
	\phi_m=-\frac{\phi}{\alpha\delta},
	\qquad
	\phi_{mm}
	=\frac1\alpha\left(\frac1\alpha-1\right)
	\frac{\phi}{\delta^2},
	\]
	the equality $X_m=0$ gives
	\[
	0=\phi_mz+\phi z_m
	=-\frac{\phi z}{\alpha\delta}+\phi z_m,
	\]
	and hence
	\begin{equation}
		\label{eq:weighted-max-zm}
		z_m=\frac{z}{\alpha\delta}.
	\end{equation}
	The inequality $X_{mm}\le0$ gives
	\[
	\phi z_{mm}+2\phi_mz_m+\phi_{mm}z\le0.
	\]
	Using \eqref{eq:weighted-max-zm}, we obtain
	\[
	\begin{aligned}
	\phi z_{mm}
	&\le
	\left[
	\frac{2}{\alpha^2}
	-\frac1\alpha\left(\frac1\alpha-1\right)
	\right]
	\frac{\phi z}{\delta^2}\\
	&=\frac1\alpha\left(1+\frac1\alpha\right)
	\frac{\phi z}{\delta^2}.
	\end{aligned}
	\]
	Therefore
	\begin{equation}
		\label{eq:weighted-max-zmm}
		z_{mm}
		\le\frac1\alpha\left(1+\frac1\alpha\right)
		\frac{z}{\delta^2}.
	\end{equation}

	Multiply \eqref{eq:reciprocal-mass-equation} by $\phi$. We have
	\[
	\begin{aligned}
		\left.X_t\right|_m
		={}&
		\phi\Y-\phi F_2(\rho)+2\V X\\
		&+\phi\alpha r^4z^{-\alpha-1}z_{mm}
		+4\alpha r\phi z^{-\alpha}z_m\\
		&-\phi\alpha(\alpha+1)r^4z^{-\alpha-2}z_m^2.
	\end{aligned}
	\]
	At the maximum, the second-mass-derivative term and the negative gradient-square term obey
	\[
	\begin{aligned}
	&\phi\alpha r^4z^{-\alpha-1}z_{mm}
	-\phi\alpha(\alpha+1)r^4z^{-\alpha-2}z_m^2\\
	&\quad\le
	\alpha r^4\phi z^{-\alpha}\delta^{-2}
	\left[
	\frac1\alpha\left(1+\frac1\alpha\right)
	-\frac{\alpha+1}{\alpha^2}
	\right]\\
	&\quad=0.
	\end{aligned}
	\]
	This is the critical cancellation at the exponent $1/\alpha$.  The remaining
	first-mass-derivative term is
	\[
	\begin{aligned}
	\phi\,4\alpha rz^{-\alpha}z_m
	&=4r\frac{\phi z^{1-\alpha}}{\delta}\\
	&=4rX^{1-\alpha},
	\end{aligned}
	\]
	because $\phi^\alpha=\delta$.  The remaining terms in
	\eqref{eq:reciprocal-mass-equation} now give
	\eqref{eq:weighted-reciprocal-maximum} at an interior maximum.

	It remains to check a maximum at the center without differentiating twice
	with respect to the degenerate coordinate $m$.  Smooth radial parity gives
	\[
	m(t,r)=\frac{\rho(t,0)}3r^3+O(r^5),
	\qquad
	\delta(t,r)=M+O(r^3),
	\]
	and
	\[
	z(t,r)=z(t,0)+\frac12z_{rr}(t,0)r^2+O(r^4).
	\]
	Consequently,
	\[
	X(t,r)=M^{1/\alpha}z(t,0)
	+\frac12M^{1/\alpha}z_{rr}(t,0)r^2+O(r^3).
	\]
	If $X$ has a maximum at the center, then $z_{rr}(t,0)\le0$.  The smooth
	radial limit in \eqref{eq:reciprocal-Y-equation} gives
	\[
	z_{rr}+\frac2rz_r\longrightarrow3z_{rr}(t,0)\le0,
	\qquad z_r(t,0)=0.
	\]
	Thus the full diffusion and gradient-square contribution is nonpositive.
	The term $4rX^{1-\alpha}$ in
	\eqref{eq:weighted-reciprocal-maximum} vanishes at $r=0$, and the same upper
	inequality follows at a center maximum.
\end{proof}

Define
\begin{equation}
	\label{eq:def-QX}
	Q_X(t):=\max_{0\le m\le M}X(t,m).
\end{equation}

\begin{proposition}[Uniform critical weighted reciprocal bound]
	\label{prop:uniform-weighted-reciprocal}
	Under the hypotheses of Theorem~\ref{thm:wall-independent}, there exists a
	constant $C_X<\infty$ such that
	\begin{equation}
		\label{eq:QX-bound}
		\sup_{0\le t<T}Q_X(t)\le C_X.
	\end{equation}
	Equivalently,
	\begin{equation}
		\label{eq:weighted-mass-density-lower}
		\rho(t,m)\ge\frac{(M-m)^{1/\alpha}}{C_X}
		\qquad(0\le t<T,\ 0\le m\le M).
	\end{equation}
\end{proposition}

\begin{proof}
We first convert the pathwise estimates of
Proposition~\ref{prop:Y-bounds} into estimates involving only $Q_X$.

Suppose that $0<\gamma-\alpha<1$.  Since $z=X/\phi$ and
$\phi=\delta^{1/\alpha}$, one has
\begin{equation}
	\label{eq:weighted-path-power}
	\phi z^{1-(\gamma-\alpha)}
	=\delta^{(\gamma-\alpha)/\alpha}
	X^{1-(\gamma-\alpha)}.
\end{equation}
Multiplying \eqref{eq:Y-path-sublinear} by the time-independent value
$\phi(m)$, using \eqref{eq:weighted-path-power} along the same fixed mass
label, and then using $0\le\delta\le M$ gives
\begin{equation}
	\label{eq:weighted-Y-sublinear}
	\phi\Y(t,m)
	\le C_T+C_T\int_0^t
	Q_X(s)^{1-(\gamma-\alpha)}\dd s.
\end{equation}
In the same range $F_2<0$, and hence
\begin{equation}
	\label{eq:weighted-F2-sublinear}
	\begin{aligned}
	-\phi F_2(\rho)
	&=\frac{\gamma}{\alpha(1-(\gamma-\alpha))}
	\phi z^{1-(\gamma-\alpha)}\\
	&\le C_TX^{1-(\gamma-\alpha)}.
	\end{aligned}
\end{equation}

If $\gamma-\alpha=1$, the elementary inequality
$\log(1+s)\le s$ gives
\begin{equation}
	\label{eq:weighted-path-log}
	\phi\log(1+z)\le\phi z=X.
\end{equation}
Multiplying \eqref{eq:Y-path-log} by $\phi(m)$ and following the same fixed
label therefore gives
\begin{equation}
	\label{eq:weighted-Y-log}
	\phi\Y(t,m)
	\le C_T+C_T\int_0^tQ_X(s)\dd s.
\end{equation}
Moreover, $F_2(\rho)=-(\gamma/\alpha)\log z$, so
\begin{equation}
	\label{eq:weighted-F2-log}
	-\phi F_2(\rho)
	=\frac\gamma\alpha\phi\log z
	\le\frac\gamma\alpha\phi\log_+z
	\le\frac\gamma\alpha X.
\end{equation}

Finally, if $\gamma-\alpha>1$, the pathwise uniform estimate
\eqref{eq:Y-path-uniform} and $0\le\phi\le M^{1/\alpha}$ give
\begin{equation}
	\label{eq:weighted-Y-uniform}
	\phi\Y(t,m)\le C_T.
\end{equation}
In this range $F_2\ge0$, so
\begin{equation}
	\label{eq:weighted-F2-uniform}
	-\phi F_2(\rho)\le0.
\end{equation}

We now close the maximum estimate.  Fix $T'<T$.  Positivity and classical
regularity on $[0,T']\times[0,R]$ make $X$ and its fixed-label time derivative
continuous on the compact mass rectangle, with boundary value zero.  Thus
Lemma~\ref{lem:Dini} and
Proposition~\ref{prop:weighted-reciprocal-maximum} apply even when maximizing
labels are nonunique or switch in time.  Set
\begin{equation}
	\label{eq:def-weighted-H}
	H(t):=\max\{1,Q_X(t)\}.
\end{equation}
When the active maximum satisfies $Q_X(t)\ge1$, the geometric term obeys
\begin{equation}
	\label{eq:weighted-geometric-cap}
	4rX^{1-\alpha}\le4R,
\end{equation}
because $\alpha>1$.  Proposition~\ref{prop:angular-bounds} also gives
\begin{equation}
	\label{eq:weighted-angular-bound}
	2\V X\le2\V_+X\le C_TX.
\end{equation}
When $Q_X(t)<1$, the upper right Dini derivative of $H$ is zero; when
$Q_X(t)=1$, it is bounded above by the positive part of $D^+Q_X(t)$.
Consequently the following nonnegative upper estimates at active maxima
control $D^+H$ in all cases.

For $0<\gamma-\alpha<1$, equations
\eqref{eq:weighted-reciprocal-maximum},
\eqref{eq:weighted-Y-sublinear}--\eqref{eq:weighted-F2-sublinear}, and
\eqref{eq:weighted-geometric-cap}--\eqref{eq:weighted-angular-bound} give
\begin{equation}
	\label{eq:weighted-H-sublinear}
	\begin{aligned}
	D^+H(t)
	\le C_T\Bigg[
	&1+H(t)+H(t)^{1-(\gamma-\alpha)}\\
	&+\int_0^tH(s)^{1-(\gamma-\alpha)}\dd s
	\Bigg].
	\end{aligned}
\end{equation}
Since $0<1-(\gamma-\alpha)<1$, one has
$x^{1-(\gamma-\alpha)}\le1+x$ for $x\ge0$.  Hence
\begin{equation}
	\label{eq:weighted-H-Volterra}
	D^+H(t)
	\le C_T\left[
	1+H(t)+\int_0^t(1+H(s))\dd s
	\right].
\end{equation}
For $\gamma-\alpha=1$, equations
\eqref{eq:weighted-Y-log}--\eqref{eq:weighted-F2-log} yield the same estimate
\eqref{eq:weighted-H-Volterra}.  For $\gamma-\alpha>1$, equations
\eqref{eq:weighted-Y-uniform}--\eqref{eq:weighted-F2-uniform} instead give
\begin{equation}
	\label{eq:weighted-H-linear}
	D^+H(t)\le C_T(1+H(t)).
\end{equation}

In the first two cases define
\[
	G_X(t):=1+H(t)+\int_0^t(1+H(s))\dd s.
\]
After increasing $C_T$, equation \eqref{eq:weighted-H-Volterra} gives
\[
	D^+G_X(t)\le C_TG_X(t).
\]
Lemma~\ref{lem:Dini-Gronwall} bounds $G_X$ on $[0,T']$.  The same lemma,
applied to $1+H$, closes \eqref{eq:weighted-H-linear}.  The constants are
independent of $T'<T$, and therefore letting $T'\uparrow T$ proves
\eqref{eq:QX-bound}.
By \eqref{eq:def-weighted-reciprocal}, this is exactly
\eqref{eq:weighted-reciprocal-bound}, and rearranging it proves
\eqref{eq:weighted-mass-density-lower}.
The weight is differentiated only at interior maximizing labels; its singular
spatial derivative at $m=M$ is never used because
\eqref{eq:weighted-reciprocal-wall} excludes the boundary first.
\end{proof}

\begin{proof}[Proof of Theorem~\ref{thm:wall-independent}]
Proposition~\ref{prop:uniform-weighted-reciprocal} proves
\eqref{eq:weighted-reciprocal-bound}.  It remains to convert
\eqref{eq:weighted-mass-density-lower} into a lower bound
on a fixed Eulerian interior region.  Regard $\delta=M-m$ as the inward mass
coordinate and write $r=r(t,\delta)$.  Equation
\eqref{eq:reciprocal-mass-derivatives} gives
\[
	r_\delta=-\frac{z}{r^2},
\]
and hence
\[
	\frac{\dd}{\dd\delta}r^3=-3z.
\]
Since $r(t,0)=R$, integration from the boundary label $0$ to $\delta$ yields the
exact identity
\begin{equation}
	\label{eq:weighted-mass-volume}
	R^3-r(t,\delta)^3=3\int_0^\delta z(t,s)\dd s.
\end{equation}
By \eqref{eq:QX-bound},
\[
	z(t,s)\le C_Xs^{-1/\alpha}.
\]
Because $\alpha>1$,
\[
	\int_0^\delta s^{-1/\alpha}\dd s
	=\frac\alpha{\alpha-1}\delta^{(\alpha-1)/\alpha},
\]
and therefore
\begin{equation}
	\label{eq:weighted-mass-volume-upper}
	R^3-r(t,\delta)^3
	\le\frac{3\alpha C_X}{\alpha-1}
	\delta^{(\alpha-1)/\alpha}.
\end{equation}

Fix $\varepsilon\in(0,R)$ and set
\[
	D_\varepsilon:=R^3-(R-\varepsilon)^3>0.
\]
If $r(t,\delta)\le R-\varepsilon$, then
$D_\varepsilon\le R^3-r(t,\delta)^3$.  Combining this with
\eqref{eq:weighted-mass-volume-upper} gives
\begin{equation}
	\label{eq:weighted-delta-epsilon}
	\delta\ge\delta_\varepsilon
	:=\left[
	\frac{(\alpha-1)D_\varepsilon}{3\alpha C_X}
	\right]^{\alpha/(\alpha-1)}>0.
\end{equation}
Consequently \eqref{eq:weighted-mass-density-lower} implies
\[
	\rho(t,r)
	\ge\frac{\delta_\varepsilon^{1/\alpha}}{C_X}
	=:c_{T,\varepsilon}>0
\]
for $0\le t<T$ and $0\le r\le R-\varepsilon$.  This proves
\eqref{eq:wall-independent-interior-lower}.

Finally, suppose $t_j\uparrow T$ and $\rho(t_j,r_j)\to0$.  If $r_j$ did not
converge to $R$, a subsequence would satisfy $r_j\le R-\varepsilon$ for some
$\varepsilon>0$, contradicting
\eqref{eq:wall-independent-interior-lower}.  Thus $r_j\to R$.  If
$\rho(t,\cdot)\to\rho_*$ uniformly and $\rho_*$ has a zero, then for every
fixed $r<R$, estimate \eqref{eq:wall-independent-interior-lower}, with for
example $\varepsilon=(R-r)/2$, gives $\rho_*(r)>0$.  Hence the assumed
nonempty zero set can contain only $R$, which proves
\eqref{eq:wall-independent-zero-set}.
\end{proof}

\section{Negative boundary trace: formation of boundary vacuum in finite time}

We now construct the family in Theorem~\ref{thm:bad-wall}.  The small
parameter $k$ sets the initial boundary scale $\rho_{0,k}(R)^\alpha\simeq k$, and the maximal time scale is $k^{(\alpha-1)/\alpha}$.  After obtaining a uniform bound for the transverse strain rate functions over a short time, we pass to inward mass coordinates, in which $W=k^{-1}\rho^\alpha$ and the physical mass flux $F=r^2u$ satisfy a degenerate scalar parabolic equation.  The data create a negative flux plateau; an expanding observation collar preserves that flux long enough to contradict the finite reciprocal capacity of an affine boundary profile.  The final subsection constructs the terminal density trace and proves that its only zero is the boundary.

\subsection{Control of transverse strain rate functions over a short time}

\begin{proposition}[Bound for transverse strain rate functions on a short time interval]
	\label{prop:normalized}
	Let $\alpha>1$, $\gamma>\alpha$, and let $(\rho,u)$ be a positive radial
	classical solution of \eqref{eq:ns-system} on $[0,T)\times B_R$.  Fix
	$t_0\in[0,T)$ and suppose that
	\[
		0<\rho\le\overline\rho
		\qquad\text{on }[t_0,T)\times B_R.
	\]
	Set
	\[
	X_{t_0}=1+\norm{\mathscr U(t_0)}_\infty
	+\norm{\mathscr V(t_0)}_\infty,
	\qquad
	\tau_{\rm ang}=\frac{1}{2C_{\overline\rho}X_{t_0}}.
	\]
	Here $C_{\overline\rho}$ depends only on $\alpha$, $\gamma$, and the density upper bound $\overline\rho$.  Then
	\[
	\norm{\mathscr U(t)}_\infty+\norm{\mathscr V(t)}_\infty
	\le2X_{t_0},
	\qquad t_0\le t<\min\{T,t_0+\tau_{\rm ang}\}.
	\]
\end{proposition}

\begin{proof}
	We use an argument based on the maximum principle.  Radial smoothness makes $\mathscr U$ an even function at the center, so $\mathscr U_r(t,0)=0$.  If a spatial maximum or minimum of $\mathscr U$ occurs at $r=0$, then the limiting value of its radial diffusion operator is
	\[
	\lim_{r\downarrow0}\left(
	\mathscr U_{rr}+\frac4r\mathscr U_r\right)
	=5\mathscr U_{rr}(t,0)\le (\ge)0.
	\]
	The same diffusion operator is nonpositive at an interior maximum and
	nonnegative at an interior minimum.  A nonzero extremum of $\abs{\mathscr U}$ cannot occur at the outer boundary because the boundary condition gives $\mathscr U(t,R)=0$.  Thus the diffusion has the favorable sign at every spatial record relevant to $\norm{\mathscr U(t)}_\infty$.
	
	Since $\gamma-\alpha>0$, the density upper bound implies
	\[
	0\le F_1(\rho)
	\le\frac\gamma\alpha\overline\rho^{\gamma-\alpha}.
	\]
	Define
	\[
	M_+(t)=\max_{0\le r\le R}\mathscr U(t,r),
	\qquad
	M_-(t)=-\min_{0\le r\le R}\mathscr U(t,r).
	\]
	Because $\mathscr U(t,R)=0$, both $M_+(t)$ and $M_-(t)$ are nonnegative, and
	\[
	\norm{\mathscr U(t)}_\infty=\max\{M_+(t),M_-(t)\}.
	\]
	
	First suppose that $r_+$ is a point at which a positive maximum $M_+(t)$ is attained.  At a center or interior record, $\mathscr U_r(t,r_+)=0$, so the transport part of the material derivative vanishes:
	\[
	\Dt\mathscr U(t,r_+)
	=\partial_t\mathscr U(t,r_+).
	\]
	The diffusion term in \eqref{eq:normalized-u-equation} is nonpositive there by the preceding calculation.  The term $\alpha r\mathscr U_r$ also vanishes.  Therefore, evaluating \eqref{eq:normalized-u-equation} at $r_+$ gives
	\[
	\begin{aligned}
		\partial_t\mathscr U(t,r_+)
		&\le
		\bigl(\mathscr V(t,r_+)-\mathscr U(t,r_+)\bigr)
		\bigl((3\alpha-2)\mathscr U(t,r_+)-F_1(\rho(t,r_+))\bigr)
		-\mathscr U(t,r_+)^2
		\\
		&\le
		\left(\norm{\mathscr V(t)}_\infty
		+\norm{\mathscr U(t)}_\infty\right)
		\left((3\alpha-2)\norm{\mathscr U(t)}_\infty
		+\frac\gamma\alpha\overline\rho^{\gamma-\alpha}\right)
		\\
		&\le C_{\overline\rho}
		\left(1+\norm{\mathscr U(t)}_\infty
		+\norm{\mathscr V(t)}_\infty\right)^2.
	\end{aligned}
	\]
	
	Next suppose that $r_-$ is a point at which a negative minimum is attained, so $M_-(t)=-\mathscr U(t,r_-)$.  Again $\Dt\mathscr U(t,r_-)=\partial_t\mathscr U(t,r_-)$ and $\mathscr U_r(t,r_-)=0$.  At this minimum the diffusion term in the equation for $\mathscr U$ is nonnegative.  After multiplying the equation by $-1$ to estimate $-\mathscr U$, its diffusion contribution is nonpositive.  Hence
	\[
	\begin{aligned}
		-\partial_t\mathscr U(t,r_-)
		&\le
		\abs{\mathscr V(t,r_-)-\mathscr U(t,r_-)}
		\left((3\alpha-2)\abs{\mathscr U(t,r_-)}
		+F_1(\rho(t,r_-))\right)
		+\abs{\mathscr U(t,r_-)}^2
		\\
		&\le C_{\overline\rho}
		\left(1+\norm{\mathscr U(t)}_\infty
		+\norm{\mathscr V(t)}_\infty\right)^2.
	\end{aligned}
	\]
	
	Applying Lemma~\ref{lem:Dini} to $\mathscr U$ and $-\mathscr U$, respectively, gives
	\[
	\begin{aligned}
		D^+M_+(t)
		&\le\max_{r\in\operatorname{argmax}\mathscr U(t,\cdot)}
		\partial_t\mathscr U(t,r),\\
		D^+M_-(t)
		&\le\max_{r\in\operatorname{argmin}\mathscr U(t,\cdot)}
		\bigl(-\partial_t\mathscr U(t,r)\bigr).
	\end{aligned}
	\]
	The preceding two pointwise estimates, followed by $\norm{\mathscr U}_\infty=\max\{M_+,M_-\}$, therefore give
	\[
	D^+\norm{\mathscr U(t)}_\infty
	\le C_{\overline\rho}
	\left(1+\norm{\mathscr U(t)}_\infty
	+\norm{\mathscr V(t)}_\infty\right)^2.
	\]
	On the other hand, along each material trajectory, \eqref{eq:angular-V} gives
	\[
	\frac{\dd}{\dd t}\abs{\mathscr V}
	\le \abs{\mathscr U}\abs{\mathscr V}
	+F_1(\rho)\left(\abs{\mathscr V}+\abs{\mathscr U}\right).
	\]
	Taking the supremum over particle paths gives
	\[
	D^+\norm{\mathscr V(t)}_\infty
	\le C_{\overline\rho}
	\left(1+\norm{\mathscr U(t)}_\infty
	+\norm{\mathscr V(t)}_\infty\right)^2.
	\]
	Consequently, with
	\[
	X(t)=1+\norm{\mathscr U(t)}_\infty+\norm{\mathscr V(t)}_\infty,
	\]
	we obtain
	\begin{equation}
		\label{eq:Riccati}
		D^+X(t)\le C_{\overline\rho}X(t)^2.
	\end{equation}
	The comparison solution with initial value $X_{t_0}$ at time $t_0$ is
	\[
	Y(t)=\frac{X_{t_0}}
	{1-C_{\overline\rho}X_{t_0}(t-t_0)}.
	\]
	For $t_0\le t<t_0+\tau_{\rm ang}$ its denominator is larger than $1/2$, so
	\[
	X(t)\le Y(t)\le2X_{t_0}.
	\]
	This proves the asserted bound.
\end{proof}

\subsection{boundary scaling and flux equations}
Set
\begin{equation}
	\label{eq:beta-nu}
	\beta=\frac{\alpha-1}{\alpha},
	\qquad \nu=1+\frac1\alpha.
\end{equation}
For $k>0$, introduce the rescaled time $\theta$ and the rescaled inward mass coordinate $\zeta$ by
\[
t=k^\beta\theta, \qquad \delta=k\zeta.
\]
Define the pullback functions
\begin{equation}
	\label{eq:scaling}
	\begin{gathered}
		\tilde{r}(\theta,\zeta)
		:=r(k^\beta\theta,k\zeta), \qquad
		\tilde\rho(\theta,\zeta)
		:=\rho\bigl(k^\beta\theta,\tilde{r}(\theta,\zeta)\bigr),\\
		\tilde{u}(\theta,\zeta)
		:=u\bigl(k^\beta\theta,\tilde{r}(\theta,\zeta)\bigr), \qquad
		\tilde{v}(\theta,\zeta)
		:=v\bigl(k^\beta\theta,\tilde{r}(\theta,\zeta)\bigr).
	\end{gathered}
\end{equation}
Define
\begin{equation}
	\label{eq:H-def}
	\begin{aligned}
		W(\theta,\zeta)
		&:=k^{-1}\tilde\rho^\alpha(\theta,\zeta),\\
		F(\theta,\zeta)
		&:=\tilde{r}^2(\theta,\zeta)\tilde{u}(\theta,\zeta), \\
		H(\theta,\zeta)
		&:=\tilde{r}^2(\theta,\zeta)\tilde{v}(\theta,\zeta).
	\end{aligned}
\end{equation}
A direct calculation gives
\begin{equation}
	\label{eq:scaled-kinematic}
	\begin{aligned}
		\tilde{r}_\zeta&=-k^\beta W^{-1/\alpha}\tilde{r}^{-2},
		\qquad
		\tilde{r}_\theta=k^\beta\frac{F}{\tilde{r}^2}.
	\end{aligned}
\end{equation}
The effective flux satisfies
\begin{equation}
	\label{eq:H-evolution}
	H_\theta=k^\beta\left(\frac{2F H}{\tilde{r}^3}
	+F_1(\tilde{\rho})(F-H)\right).
\end{equation}
\begin{equation}
	\label{eq:scaled-flux}
	F_\theta-\alpha \tilde{r}^4(W^\nu F_\zeta)_\zeta=k^\beta\left(\frac{2F(2F-H)}{\tilde{r}^3}
	+F_1(\tilde{\rho})(F-H)\right):=\Scal.
\end{equation}
Since
\[
F(\theta,\zeta)=\tilde{r}^3(\theta,\zeta)\mathscr U(k^\beta\theta,\tilde{r}(\theta,\zeta)),
\qquad H=\tilde{r}^3(\theta,\zeta)\mathscr V(k^\beta\theta,\tilde{r}(\theta,\zeta)),
\]
the source has no center singularity.  The density upper bound and
Proposition~\ref{prop:normalized} give, on every bounded interval of rescaled time,
\begin{equation}
	\label{eq:source-small}
	\abs{\Scal}\le Ck^\beta.
\end{equation}
The no-slip boundary condition and center parity give
\begin{equation}
	\label{eq:Phi-boundary}
	F(\theta,0)=0,
	\qquad
	F\left(\theta,\frac Mk\right)=0.
\end{equation}

\begin{lemma}[Global range for the scaled flux]
	\label{lem:global-range}
	Fix an endpoint $\Theta>0$ in rescaled time.  Suppose that a classical
	solution is defined on $[0,T)$, that
	\[
		0<\rho\le\overline\rho
		\quad\hbox{on}\quad
		0\le t<\min\{T,k^\beta\Theta\},
		\qquad
		k^\beta\Theta\le\tau_{\rm ang},
	\]
	and that the initial transverse strain rate functions satisfy
	\[
		X_0:=1+\norm{\mathscr U(0)}_\infty
		+\norm{\mathscr V(0)}_\infty < \infty,
	\]
	where $\overline\rho$ and $X_0$ are independent of $k$.  Suppose also that,
	for a fixed $J<0$,
	\[
		J\le F(0,\zeta)\le0
		\quad\hbox{on}\quad 0\le\zeta\le M/k.
	\]
	Then, for
	$0\le\theta<\min\{\Theta,T/k^\beta\}$,
	\begin{align}
		J-C_\Theta k^\beta\theta
		&\le F(\theta,\zeta)
		\le C_\Theta k^\beta\theta,
	\label{eq:F-range}\\
		\abs{H(\theta,\zeta)-H_0(\zeta)}
		&\le C_\Theta k^\beta\theta,
	\label{eq:H-range}\\
		\frac{\tilde r(\theta,\zeta)}{\tilde r_0(\zeta)}
		&=\exp\left(k^\beta\int_0^\theta
		\mathscr U\bigl(k^\beta s,\tilde r(s,\zeta)\bigr)\dd s\right).
	\label{eq:r-ratio}
	\end{align}
	Here
	$H_0(\zeta)=H(0,\zeta)$,
	$\tilde r_0(\zeta)=\tilde r(0,\zeta)$, and $C_\Theta$ is independent of $k$.
\end{lemma}

\begin{proof}
	Put
	\[
		\Theta_*:=\min\{\Theta,T/k^\beta\}.
	\]
	Throughout the proof we abbreviate the pullbacks to scaled mass
	coordinates by
	\[
		\mathcal U(\theta,\zeta)
		:=\mathscr U\bigl(k^\beta\theta,\tilde r(\theta,\zeta)\bigr),
		\qquad
		\mathcal V(\theta,\zeta)
		:=\mathscr V\bigl(k^\beta\theta,\tilde r(\theta,\zeta)\bigr).
	\]
	We first obtain a bound that is uniform in $k$ and does not use a density
	lower bound.  If $0\le\theta<\Theta_*$, then
	$k^\beta\theta<\min\{T,k^\beta\Theta\}\le\tau_{\rm ang}$.
	Proposition~\ref{prop:normalized} therefore gives
	\begin{equation}
		\label{eq:global-normalized-bound}
		\norm{\mathcal U(\theta)}_\infty
		+\norm{\mathcal V(\theta)}_\infty
		\le2X_0.
	\end{equation}
	Since $\gamma>\alpha$ and $0<\rho\le\overline\rho$,
	\begin{equation}
		\label{eq:global-F1-bound}
		0\le F_1(\rho)
		=\frac\gamma\alpha\rho^{\gamma-\alpha}
		\le\frac\gamma\alpha\overline\rho^{\gamma-\alpha}
		=:C_{\overline\rho}.
	\end{equation}

	For $0\le\zeta<M/k$, the identities	$F=\tilde r^3\mathcal U$ and $H=\tilde r^3\mathcal V$ give
	\begin{equation}
		\label{eq:global-source-rewrite}
		\Scal
		=k^\beta\tilde r^3\left[
		2\mathcal U(2\mathcal U-\mathcal V)
		+F_1(\tilde{\rho})(\mathcal U-\mathcal V)
		\right].
	\end{equation}
	Because $0\le\tilde r\le R$, \eqref{eq:global-normalized-bound} and
	\eqref{eq:global-F1-bound} yield
	\[
		\begin{aligned}
		\abs{2\mathcal U(2\mathcal U-\mathcal V)}
		&\le 4X_0(2X_0+2X_0)=16X_0^2,\\
		\abs{F_1(\tilde{\rho})(\mathcal U-\mathcal V)}
		&\le 2C_{\overline{\rho}}X_0.
		\end{aligned}
	\]
	Consequently,
	\begin{equation}
		\label{eq:global-source-bound}
		\abs{\Scal}\le C_Sk^\beta,
		\qquad
		C_S:=R^3\left(16X_0^2+2C_{\overline{\rho}}X_0\right).
	\end{equation}

	We next prove the upper bound for $F$.  Fix any
	$0<\bar{\theta}<\Theta_*$ and any $\varepsilon>0$, and set
	\[
		G_\varepsilon(\theta,\zeta)
		:=F(\theta,\zeta)-(C_S+\varepsilon)k^\beta\theta
	\]
	on the closed cylinder
	$[0,\bar{\theta}]\times[0,M/k]$.  By \eqref{eq:scaled-flux} and
	\eqref{eq:global-source-bound},
	\begin{equation}
		\label{eq:upper-barrier-equation}
		(G_\varepsilon)_\theta
		-\alpha\tilde r^4
		(W^\nu(G_\varepsilon)_\zeta)_\zeta
		=\Scal-(C_S+\varepsilon)k^\beta
		\le-\varepsilon k^\beta<0.
	\end{equation}
	The initial range gives $G_\varepsilon(0,\zeta)\le0$, while
	\eqref{eq:Phi-boundary} gives
	\[
		G_\varepsilon(\theta,0)
		=G_\varepsilon(\theta,M/k)
		=-(C_S+\varepsilon)k^\beta\theta\le0.
	\]
	Suppose, for contradiction, that $G_\varepsilon$ has a positive maximum
	on this cylinder.  It cannot occur at the initial time or at either spatial
	endpoint, so it occurs at some $(\theta_0,\zeta_0)$ with
	$0<\theta_0\le\bar{\theta}$ and $0<\zeta_0<M/k$.  At this point,
	\[
		(G_\varepsilon)_\theta\ge0,
		\qquad
		(G_\varepsilon)_\zeta=0,
		\qquad
		(G_\varepsilon)_{\zeta\zeta}\le0.
	\]
	Here the first inequality is the one-sided time derivative if
	$\theta_0=\bar{\theta}$.  Since $W>0$ on the positive classical branch,
	\[
		\begin{aligned}
		(W^\nu(G_\varepsilon)_\zeta)_\zeta
		&=W^\nu(G_\varepsilon)_{\zeta\zeta}
		+\nu W^{\nu-1}W_\zeta(G_\varepsilon)_\zeta\\
		&=W^\nu(G_\varepsilon)_{\zeta\zeta}\le0.
		\end{aligned}
	\]
	The left-hand side of \eqref{eq:upper-barrier-equation} is therefore
	nonnegative, contradicting its strictly negative right-hand side.  Hence
	$G_\varepsilon\le0$.  Letting $\varepsilon\downarrow0$ and then taking
	$\bar{\theta}\uparrow\Theta_*$ gives
	\begin{equation}
		\label{eq:global-F-upper-proof}
		F(\theta,\zeta)\le C_Sk^\beta\theta,
		\qquad 0\le\theta<\Theta_*.
	\end{equation}

	For the lower bound define, on the same cylinder,
	\[
		L_\varepsilon(\theta,\zeta)
		:=J-F(\theta,\zeta)-(C_S+\varepsilon)k^\beta\theta.
	\]
	Then
	\begin{equation}
		\label{eq:lower-barrier-equation}
		(L_\varepsilon)_\theta
		-\alpha\tilde r^4
		(W^\nu(L_\varepsilon)_\zeta)_\zeta
		=-\Scal-(C_S+\varepsilon)k^\beta
		\le-\varepsilon k^\beta<0.
	\end{equation}
	The initial range gives
	$L_\varepsilon(0,\zeta)=J-F(0,\zeta)\le0$.  At both spatial endpoints,
	\[
		L_\varepsilon
		=J-(C_S+\varepsilon)k^\beta\theta<0
	\]
	because $J<0$.  The same argument at an interior maximum, applied to
	\eqref{eq:lower-barrier-equation} gives $L_\varepsilon\le0$.  Letting
	$\varepsilon\downarrow0$ and $\bar{\theta}\uparrow\Theta_*$ yields
	\begin{equation}
		\label{eq:global-F-lower-proof}
		F(\theta,\zeta)\ge J-C_Sk^\beta\theta,
		\qquad 0\le\theta<\Theta_*.
	\end{equation}
	Together, \eqref{eq:global-F-upper-proof} and
	\eqref{eq:global-F-lower-proof} prove \eqref{eq:F-range} with any
	$C_\Theta\ge C_S$.

	It remains to estimate $H$ and integrate the kinematic equation.  From
	\eqref{eq:H-evolution},
	\begin{equation}
		\label{eq:global-H-rewrite}
		H_\theta
		=k^\beta\tilde r^3\left[
		2\mathcal U\mathcal V
		+F_1(\tilde{\rho})(\mathcal U-\mathcal V)
		\right].
	\end{equation}
	The same bounds give
	\begin{equation}
		\label{eq:global-H-derivative-bound}
		\abs{H_\theta}\le C_Hk^\beta,
		\qquad
		C_H:=R^3\left(2X_0^2+2C_{\overline{\rho}}X_0\right).
	\end{equation}
	At fixed material label $\zeta$, integration from $0$ to $\theta$ gives
	\[
		\abs{H(\theta,\zeta)-H_0(\zeta)}
		\le\int_0^\theta\abs{H_s(s,\zeta)}\dd s
		\le C_Hk^\beta\theta.
	\]
	This is \eqref{eq:H-range} after taking
	$C_\Theta:=\max\{C_S,C_H\}$.

	Finally, if $0\le\zeta<M/k$, then $\tilde r(\theta,\zeta)>0$ and the
	kinematic identity \eqref{eq:scaled-kinematic} gives
	\[
		\tilde r_\theta
		=k^\beta\frac{F}{\tilde r^2}
		=k^\beta\tilde r\,\mathcal U.
	\]
	Dividing by $\tilde r$ and integrating at fixed $\zeta$ produces
	\[
		\log\frac{\tilde r(\theta,\zeta)}{\tilde r_0(\zeta)}
		=k^\beta\int_0^\theta\mathcal U(s,\zeta)\dd s.
	\]
	Exponentiation proves \eqref{eq:r-ratio} away from the center.  The bound
	\eqref{eq:global-normalized-bound} and radial smoothness allow
	$\zeta\uparrow M/k$ in the right-hand side, which gives the continuous radial interpretation of the quotient at the center.
\end{proof}

\subsection{Compatible initial data in the boundary layer}
Fix
\begin{equation}
	\label{eq:data-constants}
	A_*>0,
	\qquad -R^4A_*<J<0,
	\qquad a_0=A_*+\frac{J}{R^4}>0,
	\qquad \rho_{\rm c}>0.
\end{equation}
The initial density and velocity constructed in this section are
\begin{equation}
	\label{eq:rho0-piecewise}
	\rho_{0,k}(r):=
	\left\{
	\begin{aligned}
		&(1-\eta(r))\rho_{\rm c}
		+\eta(r)\rho_{{\rm aff},k}(r),
		&&0\le r\le r_{2,k},\\
		&\varrho_{0,k}\bigl(\delta_{0,k}(r)\bigr),
		&&r_{2,k}\le r\le R,
	\end{aligned}
	\right.
\end{equation}
and
\begin{equation}
	\label{eq:u0-piecewise}
	u_{0,k}(r):=
	\left\{
	\begin{aligned}
		&0,
		&&r=0,\\
		&\displaystyle\frac{J}{r^2}Q_{\rm c}\left(\frac r{r_{2,k}}\right),
		&&0<r\le r_{2,k},\\
		&\displaystyle\frac{1}{r^2}
		F_0\left(\frac{\delta_{0,k}(r)}k\right),
		&&r_{2,k}\le r\le R.
	\end{aligned}
	\right.
\end{equation}
All auxiliary profiles and interface quantities in \eqref{eq:rho0-piecewise}--\eqref{eq:u0-piecewise} are defined as below.

Choose a smooth nondecreasing step $S$ on $[0,\infty)$, flat at zero, such
that
\[
S(0)=0,
\qquad S(\zeta)=1\quad(\zeta\ge L_0).
\]
Fix $\Lambda_0>0$ and choose $\chi_0\in C_c^\infty([0,L_0))$ with
$\chi_0=1$ near zero and support contained in $\{S\le1/2\}$.  Set
\begin{align}
	E_k(\zeta)
	&=-\frac{\gamma A_*}{2\alpha^2}
	\Lambda_0^{(\gamma-2\alpha-1)/\alpha}
	k^{(\gamma-1)/\alpha}\zeta^2\chi_0(\zeta),
	\label{eq:Ek}\\
	F_0(\zeta)&=JS(\zeta)+E_k(\zeta),
	\label{eq:F0}\\
	(W_0)_\zeta&=A_*+\frac{F_0}{R^4},
	\qquad W_0(0)=\Lambda_0.
	\label{eq:W0}
\end{align}
Because $(\gamma-1)/\alpha>0$, $\|E_k\|_{L^\infty}
\le Ck^{(\gamma-1)/\alpha}$.  For sufficiently small $k$,
$\|E_k\|_\infty\le(-J)/2$, and hence
\begin{equation}
	\label{eq:data-ranges}
	J\le F_0\le0,
	\qquad
	a_0\le(W_0)_\zeta\le A_*.
\end{equation}
After the fixed transition at $\zeta=L_0$, the functions $F_0$ and
$(W_0)_\zeta$ reach the constant values
\begin{equation}
	\label{eq:plateau}
	F_0=J,
	\qquad (W_0)_\zeta=a_0.
\end{equation}
We retain this plateau through a fixed interval in the physical inward mass
coordinate, $kL_0\le\delta\le\delta_2$, with $\delta_2$ independent of $k$.

Set
\begin{equation}
	\label{eq:collar-density-radius}
	\varrho_{0,k}(\delta)
	:=\left[kW_0\left(\frac\delta k\right)\right]^{1/\alpha},
\end{equation}
and recall
\[
r_{0,k}(\delta)^3
:=R^3-3\int_0^\delta\varrho_{0,k}(s)^{-1}\dd s.
\]
	Since $a_0\le(W_0)_\zeta\le A_*$ and $\alpha>1$, we have
	\begin{equation}
	\label{eq:data-collar-bounds}
	(k\Lambda_0+A_*\delta)^{-1/\alpha}
	\le \varrho_{0,k}(\delta)^{-1}
	\le(k\Lambda_0+a_0\delta)^{-1/\alpha}
	\le(a_0\delta)^{-1/\alpha}.
	\end{equation}
Thus
\[
\frac{A_*^{-1/\alpha}}{2\beta}\delta_2^\beta
\le \int_0^{\delta_2}\varrho_{0,k}(s)^{-1}\dd s
\le \frac{a_0^{-1/\alpha}}{\beta}\delta_2^\beta
\]
for all sufficiently small $k$.  We may therefore choose
$\delta_2, c_2>0$, independently of $k$, so that
\begin{equation}
	\label{eq:r2-bounds}
	r_{2,k}:=r_{0,k}(\delta_2),
	\qquad
	\frac{3R}{4}\le r_{2,k}\le R-c_2.
\end{equation}

Let $\delta_{0,k}(r)$ denote the inverse of $r_{0,k}(\delta)$ from $[0,\delta_2]$ onto $[r_{2,k},R]$, and put
\[
\rho_{2,k}:=\varrho_{0,k}(\delta_2).
\]
The plateau identity $(W_0)_\zeta=a_0$ gives
$\partial_r(\rho_0^{\alpha-1})
=-\frac{a_0(\alpha-1)}{\alpha}r^2$.  Since
$\rho_0(r_{2,k})=\rho_{2,k}$ at the interface,
\[
\rho_0(r)^{\alpha-1}-\rho_{2,k}^{\alpha-1}
=-\frac{a_0(\alpha-1)}{\alpha}
\int_{r_{2,k}}^r s^2\,ds
\]
we have the inward continuation
\begin{equation}
	\label{eq:rho-affine-r}
	\rho_{{\rm aff},k}(r)
	:=\left[
	\rho_{2,k}^{\alpha-1}
	+\frac{a_0(\alpha-1)}{3\alpha}(r_{2,k}^3-r^3)
	\right]^{1/(\alpha-1)}.
\end{equation}
Finally, choose a fixed smooth cutoff $\eta:[0,R]\to[0,1]$ and a fixed
smooth nondecreasing function $Q_{\rm c}:[0,1]\to[0,1]$ such that
\begin{equation}
	\label{eq:center-cutoffs}
	\begin{gathered}
		\eta=0\quad(0\le r\le R/3),
		\qquad
		\eta=1\quad(2R/3\le r\le R),\\
		Q_{\rm c}(s)=s^3\quad\hbox{near }s=0,
		\qquad
		Q_{\rm c}(s)=1\quad\hbox{near }s=1.
	\end{gathered}
\end{equation}

Let
$\mathcal L_{0,k}:=\mathcal L_0(\rho_{0,k},\boldsymbol u_{0,k})$, with
$\mathcal L_0$ defined in \eqref{eq:L0}.

\begin{lemma}[Compatibility of initial data]
	\label{lem:data}
	For every fixed sufficiently small $k>0$, the pair
	$(\rho_{0,k},\boldsymbol u_{0,k})$ constructed above is smooth, radial, and
	has strictly positive density.  It satisfies the boundary condition and the
	center parity conditions.  Furthermore,
	\begin{equation}
		\label{eq:data-compatibility}
		\mathcal G_{\alpha,\gamma}
		(\rho_{0,k},\boldsymbol u_{0,k})\in H_0^1(B_R).
	\end{equation}
	Moreover, there are constants independent of small $k$ such that
	\begin{equation}
		\label{eq:uniform-data}
		\sup_{0<k<k_0}\left(
		\mathcal L_{0,k}
		+\norm{u_{0,k}/r}_{L^\infty}
		+\norm{v_{0,k}/r}_{L^\infty}
		\right)<\infty,
		\qquad
		0<M_-\le M_k\le M_+<\infty,
	\end{equation}
	Finally,
	\begin{equation}
		\label{eq:wall-values}
		\rho_{0,k}(R)=(k\Lambda_0)^{1/\alpha}>0,
		\qquad
		v_{0,k}(R)=-R^2A_*<0.
	\end{equation}
\end{lemma}

\begin{proof}
	We divide the proof into four steps.

	\smallskip
	\noindent
	\emph{Step 1: matching, smoothness, positivity, and center regularity.}
	On the collar $0\le\delta\le\delta_2$, the definition of $\varrho_{0,k}$ gives
	\begin{equation}
		\label{eq:data-collar-power}
		\rho_{0,k}(r)^\alpha
		=kW_0\left(\frac{\delta_{0,k}(r)}k\right),
		\qquad
		(\delta_{0,k})_r=-\rho_{0,k}r^2.
	\end{equation}
	On the plateau $\delta/k\ge L_0$, differentiating
	\eqref{eq:data-collar-power} and using
	$(W_0)_\zeta=a_0$ gives
	\[
	\partial_r\bigl(\rho_{0,k}^{\alpha-1}\bigr)
	=-\frac{a_0(\alpha-1)}{\alpha}r^2.
	\]
	Thus the collar density agrees, on a full neighborhood of $r=r_{2,k}$,
	with $\rho_{{\rm aff},k}$ from \eqref{eq:rho-affine-r}.  Because
	$r_{2,k}\ge3R/4>2R/3$, one also has $\eta=1$ on that neighborhood.  Hence
	the two formulas in \eqref{eq:rho0-piecewise} match to every order at
	$r=r_{2,k}$.

	Likewise, $F_0=J$ on a neighborhood of $\delta=\delta_2$, whereas
	$Q_{\rm c}=1$ on a neighborhood of $r/r_{2,k}=1$.  The two velocity
	formulas in \eqref{eq:u0-piecewise} therefore also match to every order.
	Near the center, the cutoff properties give
	\[
		\rho_{0,k}(r)=\rho_{\rm c},
		\qquad
		u_{0,k}(r)=\frac{J}{r^2}
		\left(\frac r{r_{2,k}}\right)^3
		=\frac{J}{r_{2,k}^3}r.
	\]
	Consequently, $\rho_{0,k}(|x|)$ is constant near $x=0$, and
	\[
		\boldsymbol u_{0,k}(x)
		=u_{0,k}(|x|)\frac{x}{|x|}
		=\frac{J}{r_{2,k}^3}x
	\]
	there.  This verifies the required even scalar parity and odd radial vector
	parity at the center.  It follows that both fields are smooth on the closed
	ball for every fixed $k>0$.

	Finally, $S$ is flat at zero and $E_k(0)=0$, so $F_0(0)=0$.  Therefore
	\[
		u_{0,k}(R)=\frac{F_0(0)}{R^2}=0,
	\]
	which is the boundary condition.

	\smallskip
	\noindent
	\emph{Step 2: the boundary compatibility condition.}
	Use the equation acceleration defined in
	\eqref{eq:equation-acceleration}.
	If
	\[
		d=u_r+\frac{2u}{r}=\frac{(r^2u)_r}{r^2},
	\]
	then the radial coefficient of $\mathcal G_{\alpha,\gamma}$ is
	\begin{equation}
		\label{eq:data-radial-acceleration}
		G_{\alpha,\gamma}^{\rm rad}
		=\rho^{-1}\left[
		(\alpha\rho^\alpha d)_r
		-\frac2r(\rho^\alpha)_ru
		-(\rho^\gamma)_r
		\right]-uu_r.
	\end{equation}
	At a boundary at which $u(R)=u_r(R)=0$, one has
	$d(R)=0$ and
	\[
		d_r(R)
		=u_{rr}(R)+\frac2R u_r(R)-\frac2{R^2}u(R)
		=u_{rr}(R).
	\]
	Evaluation of \eqref{eq:data-radial-acceleration} at $r=R$ therefore gives
	\[
		G_{\alpha,\gamma}^{\rm rad}(R)
		=\alpha\rho^{\alpha-1}(R)u_{rr}(R)
		-\gamma\rho^{\gamma-2}(R)\rho_r(R).
	\]
	Hence the zero trace of the initial acceleration is equivalent to the scalar
	identity
	\begin{equation}
		\label{eq:data-scalar-compatibility}
		u_{rr}(R)
		=\frac\gamma\alpha
		\rho^{\gamma-\alpha-1}(R)\rho_r(R).
	\end{equation}

	We now compute both sides of
	\eqref{eq:data-scalar-compatibility} for the constructed data.  Since
	\[
		\rho_{0,k}(R)
		=(k\Lambda_0)^{1/\alpha}.
	\]
	When $r_{2,k}\le r\le R$, with $\zeta=\delta_{0,k}/k$, one has
	\[
		\zeta_r=-\frac{\rho_{0,k}r^2}{k},
		\qquad
		u_{0,k}=\frac{F_0(\zeta)}{r^2}.
	\]
	Consequently,
	\begin{equation}
		\label{eq:data-ur}
		(u_{0,k})_r
		=\frac{(F_0)_\zeta\zeta_r}{r^2}-\frac{2F_0}{r^3}
		=-\frac{\rho_{0,k}}k(F_0)_\zeta-\frac{2F_0}{r^3}.
	\end{equation}
	The flatness of $S$ at zero and the factor $\zeta^2$ in $E_k$ imply
	\[
		F_0(0)=(F_0)_\zeta(0)=0.
	\]
	Thus
	$u_{0,k}(R)=(u_{0,k})_r(R)=0$, as required in the reduction above.
	Differentiating \eqref{eq:data-ur} once more gives
	\[
	\begin{aligned}
		(u_{0,k})_{rr}
		={}&-\frac{(\rho_{0,k})_r}{k}(F_0)_\zeta
		-\frac{\rho_{0,k}}k(F_0)_{\zeta\zeta}\zeta_r\\
		&-\frac{2(F_0)_\zeta\zeta_r}{r^3}
		+\frac{6F_0}{r^4}.
	\end{aligned}
	\]
	At $r=R$, all terms except the second one vanish, and hence
	\begin{equation}
		\label{eq:data-urr-wall}
		(u_{0,k})_{rr}(R)
		=\frac{\rho_{0,k}^2(R)R^2}{k^2}
		(F_0)_{\zeta\zeta}(0).
	\end{equation}
	Since $S$ is flat at zero and $\chi_0=1$ near zero,
	\eqref{eq:Ek} gives
	\[
		(F_0)_{\zeta\zeta}(0)
		=-\frac{\gamma A_*}{\alpha^2}
		\Lambda_0^{(\gamma-2\alpha-1)/\alpha}
		k^{(\gamma-1)/\alpha}.
	\]
	Substituting this into \eqref{eq:data-urr-wall} and using
	$\rho_{0,k}(R)=(k\Lambda_0)^{1/\alpha}$, we find
	\begin{equation}
		\label{eq:data-urr-value}
	\begin{aligned}
		(u_{0,k})_{rr}(R)
		&=-\frac{\gamma A_*R^2}{\alpha^2}
		\Lambda_0^{(\gamma-2\alpha+1)/\alpha}
		k^{(\gamma-2\alpha+1)/\alpha}\\
		&=-\frac{\gamma A_*R^2}{\alpha^2}
		\rho_{0,k}^{\gamma-2\alpha+1}(R).
	\end{aligned}
	\end{equation}

	On the other hand, differentiating
	\eqref{eq:data-collar-power} in $r$ gives the identity
	\begin{equation}
		\label{eq:data-rhor-identity}
		(\rho_{0,k})_r
		=-\frac{r^2}{\alpha}(W_0)_\zeta
		\rho_{0,k}^{2-\alpha}.
	\end{equation}
	Because $F_0(0)=0$, equation \eqref{eq:W0} yields
	$(W_0)_\zeta(0)=A_*$.  Therefore
	\begin{equation}
		\label{eq:data-rhor-wall}
		(\rho_{0,k})_r(R)
		=-\frac{A_*R^2}{\alpha}
		\rho_{0,k}^{2-\alpha}(R).
	\end{equation}
	It follows term by term that
	\[
	\begin{aligned}
		\frac\gamma\alpha
		\rho_{0,k}^{\gamma-\alpha-1}(R)
		(\rho_{0,k})_r(R)
		&=-\frac{\gamma A_*R^2}{\alpha^2}
		\rho_{0,k}^{\gamma-2\alpha+1}(R)\\
		&=(u_{0,k})_{rr}(R),
	\end{aligned}
	\]
	which proves \eqref{eq:data-scalar-compatibility}.

	For fixed $k$, the data are smooth and have a positive minimum on the closed
	ball.  Hence
	$\mathcal G_{\alpha,\gamma}(\rho_{0,k},\boldsymbol u_{0,k})$
	is a smooth radial vector field and in particular belongs to $H^1(B_R)$.
	The preceding calculation says exactly that its trace on $\partial B_R$
	vanishes.  By the trace characterization of $H_0^1(B_R)$, this proves
	\eqref{eq:data-compatibility}.

	\smallskip
	\noindent
	\emph{Step 3: uniform pointwise and mass bounds.}
	The collar estimate \eqref{eq:data-collar-bounds}, the fixed length
	$\delta_2$, and the bounds $r_{2,k}\in[3R/4,R-c_2]$ imply
	\[
		\sup_{0<k<k_0}
		\norm{\rho_{0,k}}_{L^\infty(r_{2,k},R)}<\infty.
	\]
	Moreover,
	\[
		(a_0\delta_2)^{1/\alpha}
		\le\rho_{2,k}
		\le(k_0\Lambda_0+A_*\delta_2)^{1/\alpha}.
	\]
	Thus the quantity inside the brackets in
	\eqref{eq:rho-affine-r} stays in a fixed compact subinterval of
	$(0,\infty)$ for $0\le r\le r_{2,k}$.  Since $\eta$ is fixed, the center
	extension and its first derivative are consequently bounded
	uniformly in $k$.  In particular,
	\begin{equation}
		\label{eq:data-rho-r-uniform}
		\sup_{0<k<k_0}
		\left(
		\norm{\rho_{0,k}}_{L^\infty(B_R)}
		+\norm{(\rho_{0,k})_r}_{L^\infty(0,r_{2,k})}
		\right)<\infty.
	\end{equation}

	When $r_{2,k}\le r \le R$, \eqref{eq:data-rhor-identity} gives
	\[
		v_{0,k}
		=u_{0,k}+\alpha\rho_{0,k}^{\alpha-2}(\rho_{0,k})_r
		=\frac{F_0}{r^2}-r^2(W_0)_\zeta.
	\]
	Therefore, using $r\ge3R/4$, $\abs{F_0}\le\abs J$, and
	$0<(W_0)_\zeta\le A_*$,
	\begin{equation}
		\label{eq:data-normalized-collar}
		\left|\frac{u_{0,k}}r\right|
		\le\abs J\left(\frac4{3R}\right)^3,
		\qquad
		\left|\frac{v_{0,k}}r\right|
		\le\abs J\left(\frac4{3R}\right)^3+RA_*.
	\end{equation}
	On the center extension, put $s_r=r/r_{2,k}$.  Then
	\[
		\frac{u_{0,k}(r)}r
		=\frac{J}{r_{2,k}^3}\frac{Q_{\rm c}(s_r)}{s_r^3}.
	\]
	The quotient $Q_{\rm c}(s)/s^3$ is smooth and bounded on $[0,1]$, with its value
	understood as $1$ near zero, and $r_{2,k}\ge3R/4$.  Hence $u_{0,k}/r$
	is uniformly bounded there.  The density is constant on $0\le r\le R/3$;
	on the remaining center annulus, $r$ is bounded away from zero and the
	density and its first derivative are uniformly bounded, with the density
	also bounded away from zero.  It follows that
	\[
		\frac{\alpha\rho_{0,k}^{\alpha-2}(\rho_{0,k})_r}{r}
	\]
	is uniformly bounded on the center extension.  Since
	$v_{0,k}/r=u_{0,k}/r+\alpha\rho_{0,k}^{\alpha-2}
	(\rho_{0,k})_r/r$, this fact and
	\eqref{eq:data-normalized-collar} prove the two bounds for the transverse strain rate functions in \eqref{eq:uniform-data}.

	Finally, $\rho_{0,k}=\rho_{\rm c}$ on $0\le r\le R/3$.  Recalling that
	$M_k=\int_0^R\rho_{0,k}(r)r^2\dd r$, we obtain
	\[
		M_k\ge\rho_{\rm c}\int_0^{R/3}r^2\dd r
		=\frac{\rho_{\rm c}R^3}{81}=:M_->0.
	\]
	The uniform upper density bound gives
	\[
		M_k\le
		\frac{R^3}{3}
		\sup_{0<k<k_0}\norm{\rho_{0,k}}_{L^\infty(B_R)}
		=:M_+<\infty.
	\]

	\smallskip
	\noindent
	\emph{Step 4: the uniform functional involving only low order quantities.}
	We spell this out because the small boundary density must not be replaced by a uniform positive lower bound.  For any fixed finite exponents $p,q\ge1$, put
	\[
		s=\alpha-1+\frac1{2q}.
	\]
	The corresponding quantity in the statement is
	\[
	\begin{aligned}
		\mathcal L_{0,k}:={}&
		\int_{B_R}\left(
		\rho_{0,k}+\rho_{0,k}^\gamma
		+\rho_{0,k}\abs{\boldsymbol u_{0,k}}^2
		+\rho_{0,k}\abs{\boldsymbol u_{0,k}}^{2p}
		+\abs{\nabla\rho_{0,k}^{\alpha-1/2}}^2
		\right)\dd x\\
		&+\int_{B_R}\left(
		\rho_{0,k}\abs{v_{0,k}}^2
		+\rho_{0,k}\abs{v_{0,k}}^{2q}
		+\abs{\nabla\rho_{0,k}^s}^{2q}
		\right)\dd x.
	\end{aligned}
	\]
	For every exponent $\ell>0$, on $r_{2,k}\le r \le R$ the collar identity \eqref{eq:data-rhor-identity} gives
	\begin{equation}
		\label{eq:data-power-gradient}
		\abs{\partial_r\rho_{0,k}^\ell}
		=\frac\ell\alpha r^2\abs{(W_0)_\zeta}
		\rho_{0,k}^{\ell+1-\alpha}.
	\end{equation}
	Taking first $\ell=\alpha-1/2$ and then
	$\ell=s=\alpha-1+1/(2q)$ gives
	\[
		\abs{\partial_r\rho_{0,k}^{\alpha-1/2}}^2
		\le C\rho_{0,k},
		\qquad
		\abs{\partial_r\rho_{0,k}^{s}}^{2q}
		\le C\rho_{0,k}
	\]
	with $C$ independent of $k$.  On the center extension, the
	density is uniformly bounded above and away from zero and its first
	derivative is uniformly bounded; hence the same two quantities are uniformly
	bounded there.  Their integrals over the fixed ball are therefore bounded
	independently of $k$.

	The pointwise estimates already proved also give
	\[
		\rho_{0,k}^\gamma
		+\rho_{0,k}\abs{\boldsymbol u_{0,k}}^2
		+\rho_{0,k}\abs{\boldsymbol u_{0,k}}^{2p}
		+\rho_{0,k}\abs{v_{0,k}}^2
		+\rho_{0,k}\abs{v_{0,k}}^{2q}
		\le C
	\]
	on $B_R$.  Integration proves
	$\sup_{0<k<k_0}\mathcal L_{0,k}<\infty$, completing
	\eqref{eq:uniform-data}.

	It remains only to record the boundary values.  The first follows directly from
	$W_0(0)=\Lambda_0$:
	\[
		\rho_{0,k}(R)^\alpha=k\Lambda_0.
	\]
	For the second, use the collar formula for $v_{0,k}$, together with
	$F_0(0)=0$ and $(W_0)_\zeta(0)=A_*$, to obtain
	\[
		v_{0,k}(R)
		=\frac{F_0(0)}{R^2}-R^2(W_0)_\zeta(0)
		=-R^2A_*<0.
	\]
	This proves \eqref{eq:wall-values} and finishes the proof.
\end{proof}

\begin{proposition}[Uniform end time estimates]
	\label{prop:uniform-density-upper}
	Assume \eqref{eq:exponent-range}.
	Let $(\rho_k,\boldsymbol u_k)$ be the maximal positive classical branch
	issued from the data of Lemma~\ref{lem:data}, with lifespan $T_k^*$.
	For every fixed $T_0>0$, there is a constant $C(T_0)$, independent of
	sufficiently small $k$, such that
	\begin{equation}
		\label{eq:uniform-HMZ-family-estimates}
		\sup_{0<k<k_0}\ \sup_{0\le t<\min\{T_k^*,T_0\}}
		\left(
		\norm{\rho_k(t)}_{L^\infty(B_R)}
		+\norm{\nabla\rho_k^s(t)}_{L^{2q}(B_R)}
		+\int_{B_R}\rho_k\abs{\boldsymbol u_k}^{2p}\dd x
		\right)
		\le C(T_0).
	\end{equation}
	In particular, for a constant $\overline\rho_{T_0}<\infty$ independent of
	small $k$,
	\begin{equation}
		\label{eq:uniform-density-upper}
		\sup_{0<k<k_0}\ \sup_{0\le t<\min\{T_k^*,T_0\}}
		\norm{\rho_k(t)}_{L^\infty(B_R)}
		\le\overline\rho_{T_0}.
	\end{equation}
\end{proposition}

\begin{proof}
	Lemma~\ref{lem:data} gives
	$\sup_{0<k<k_0}\mathcal L_{0,k}<\infty$.  Apply
	Proposition~\ref{prop:bounded-time} to each branch with fixed $k$.  The
	dependence stated in \eqref{eq:bounded-time-estimates} is only through
	the common time bound and $\mathcal L_{0,k}$, so its constant is uniform
	in small $k$.  This proves \eqref{eq:uniform-HMZ-family-estimates}, and
	\eqref{eq:uniform-density-upper} follows immediately.
\end{proof}

\subsection{Persistence of negative flux on an observation collar}
We choose a growing observation range.  Let
\begin{equation}
	\label{eq:chi-range}
	\frac{\alpha+1}{2\alpha}<\chi<1,
	\qquad K_k=k^{-\chi}.
\end{equation}
For $k_0$ sufficiently small we have 
\[
2k K_k <\delta_2 \qquad I_k:=\left[\frac12K_k,\frac32K_k\right]\subset\left[L_0,\frac{\delta_2}{k}\right]
\]

\begin{proposition}[Negative observation flux]
	\label{prop:negative-observation-flux}
	Assume \eqref{eq:exponent-range}, and fix $\Theta>0$.  After reducing $k_0$ if
	necessary, for every
	$0<k<k_0(\Theta)$ and
	\[
		0\le\theta<\min\{\Theta,T_k^*/k^\beta\},
	\]
	\begin{equation}
		\label{eq:negative-observation-flux}
		-F(\theta,K_k)\ge\frac{-J}{2}>0.
	\end{equation}
\end{proposition}

\begin{proof}
	We divide the proof into four steps.
	
	\smallskip
	\noindent
	\emph{Step 1: geometric control of the observation collar.}
	We first show that $\tilde r_0$ and $\tilde r$ remain uniformly away from
	zero for $0\le\zeta\le2K_k$ and
	$\theta<\min\{\Theta,T_k^*/k^\beta\}$.
	
	The density upper bound from
	Proposition~\ref{prop:uniform-density-upper}, the uniform initial bounds of
	Lemma~\ref{lem:data}, and Proposition~\ref{prop:normalized} give a common
	control time $\tau_{\rm ang}>0$ for the transverse strain rate functions.  Because
	$k^\beta\Theta\to0$, we may reduce $k_0$ so that
	\[
		k^\beta\Theta\le\tau_{\rm ang}
	\]
	for every $0<k<k_0$.  Consequently, on the time interval in the
	statement, equation~\eqref{eq:source-small} holds with a constant depending on $\Theta$ but not on $k$.

	The initial construction gives $J\le F_0\le0$ for all $0<\zeta<M_k/k$: this is \eqref{eq:data-ranges} on $[r_{2,k},R]$ and follows from
	$F_0=JQ_{\rm c}(r/r_{2,k})$ on the center extension.  Hence
	Lemma~\ref{lem:global-range} gives the global estimates
	\eqref{eq:F-range}--\eqref{eq:r-ratio} on the time interval in the
	statement.

	We next estimate the radius and density slope on $[0, 2 K_k]$. From
	$W_0(\zeta)\ge\Lambda_0+a_0\zeta$ and the initial radius formula,
	\begin{equation}
		\label{eq:proof-initial-radius}
		R^3-\tilde r_0(\zeta)^3
		=3k^\beta\int_0^\zeta W_0(s)^{-1/\alpha}\dd s
		\le Ck^\beta(1+\zeta^\beta).
	\end{equation}
	Thus, on $0\le\zeta\le2K_k$,
	\[
		\abs{R-\tilde r_0(\zeta)}
		\le Ck^\beta K_k^\beta
		=Ck^{\beta(1-\chi)}=o(1).
	\]
	Together with \eqref{eq:r-ratio}, this gives
	$\tilde r_0\ge3R/4$ and $\tilde r\ge R/2$ there for all sufficiently
	small $k_0(\Theta)$.

	\smallskip
	\noindent
	\emph{Step 2: uniform positive lower bound for $W_\zeta$.}
	The identity for the effective flux is
	\begin{equation}
		\label{eq:proof-slope-identity}
		W_\zeta=\frac{F-H}{\tilde r^4}.
	\end{equation}
	On the fixed transition $0\le\zeta\le L_0$, the definition of the initial
	data gives
	\[
		H_0=F_0-\tilde r_0^4
		\left(A_*+\frac{F_0}{R^4}\right).
	\]
	Using \eqref{eq:F-range}--\eqref{eq:r-ratio} and the fact that
	$\tilde r_0=R+O(k^\beta)$ on this fixed transition, we obtain
	\[
	\begin{aligned}
		F-H
		&\ge J-F_0
		+\tilde r_0^4\left(A_*+\frac{F_0}{R^4}\right)
		-C_\Theta k^\beta\\
		&=a_0R^4+O(k^\beta)-C_\Theta k^\beta.
	\end{aligned}
	\]
	Dividing by $\tilde r^4$ gives
	\[
	\begin{aligned}
		W_\zeta
		&\ge
		\frac{a_0R^4-C_\Theta k^\beta}{\widetilde r^4}\\
		&=
		a_0\left(\frac R{\widetilde r}\right)^4
		-\frac{C_\Theta k^\beta}{\widetilde r^4}.
	\end{aligned}
	\]
	Since $\left(\frac R{\widetilde r}\right)^4=1+O_\Theta(k^\beta)$. We get
	\[
	W_\zeta\ge a_0-C_\Theta k^\beta,
	\qquad 0\le\zeta\le L_0.
	\]
	On the plateau $L_0\le\zeta\le2K_k$, one instead has
	$H_0=J-\tilde r_0^4a_0$, and hence
	\[
	\begin{aligned}
		F-H
		&\ge J-H_0-C_\Theta k^\beta\\
		&=J-\left(J-a_0\widetilde r_0^4\right)
		-C_\Theta k^\beta\\
		&=a_0\widetilde r_0^4-C_\Theta k^\beta.
	\end{aligned}
	\]
	Dividing by $\widetilde r^4$ gives
	\[
		W_\zeta
		\ge a_0\left(\frac{\tilde r_0}{\tilde r}\right)^4
		-C_\Theta k^\beta.
	\]
	Therefore
	\[
	\begin{aligned}
		W_\zeta
		&\ge
		a_0(1-C_\Theta k^\beta)-C_\Theta k^\beta\\
		&\ge a_0-C_\Theta k^\beta.
	\end{aligned}
	\]
	The same identities give the corresponding upper bound.  We have therefore
	proved
	\begin{equation}
		\label{eq:proof-slope-bound}
		a_0-\varepsilon_{k,\Theta}
		\le W_\zeta(\theta,\zeta)\le C_\Theta,
		\qquad 0\le\zeta\le2K_k,
		\qquad
		\varepsilon_{k,\Theta}\le C_\Theta k^\beta.
	\end{equation}
	In particular,
	\begin{equation}
		\label{eq:proof-slope-capacity-error}
		\varepsilon_{k,\Theta}K_k^\beta
		\le C_\Theta k^{\beta(1-\chi)}\longrightarrow0.
	\end{equation}

	Because $I_k \subset [L_0, \frac{\delta_2}{k}]$, and because
	$\varepsilon_{k,\Theta}\to0$, as $k\rightarrow0$ the positivity of $W(\theta,0)$ and \eqref{eq:proof-slope-bound} imply, for $\zeta\in I_k$,
	\begin{equation}
		\label{eq:proof-W-observation}
		W(\theta,\zeta)
		=W(\theta,0)+\int_0^\zeta W_\eta(\theta,\eta)\dd\eta
		\ge\frac{a_0}{2}\zeta\ge cK_k.
	\end{equation}
	On $I_k$ we consequently have
	$W\ge cK_k$, $\abs{W_\zeta}\le C_\Theta$, and
	$\tilde r\ge R/2$.

	\smallskip
	\noindent
	\emph{Step 3: a vanishing barrier eigenvalue.}
	Define on $I_k$
	\[
		\psi_k(\zeta)
		=\sin^2\left(\frac\pi{K_k}
		\left(\zeta-\frac12K_k\right)\right).
	\]
	Put $c_k=\pi/K_k$ and
	$Q_k=c_k\cot(c_k(\zeta-K_k/2))$.  Then
	\[
		\frac{\psi_{k,\zeta}}{\psi_k}=2Q_k,
		\qquad
		\frac{\psi_{k,\zeta\zeta}}{\psi_k}
		=2(Q_k^2-c_k^2).
	\]
	A direct calculation followed by completion of the square in $Q_k$ gives
	\begin{equation}
		\label{eq:proof-barrier-square}
		-\frac{\alpha\tilde r^4
		(W^\nu\psi_{k,\zeta})_\zeta}{\psi_k}
		\le C\left(
		W^\nu K_k^{-2}+W^{\nu-2}W_\zeta^2\right).
	\end{equation}
	Proposition~\ref{prop:uniform-density-upper} gives $W\le C/k$.  Since
	$\nu-2=-\beta<0$, the estimates on $I_k$ imply that the right-hand side of
	\eqref{eq:proof-barrier-square} is at most
	\begin{equation}
		\label{eq:proof-lambda-k}
		\lambda_k
		:=C_\Theta\left(k^{2\chi-\nu}+K_k^{-\beta}\right)
		\longrightarrow0.
	\end{equation}
	Here $2\chi>\nu$ is exactly the lower inequality in
	\eqref{eq:chi-range}.

	\smallskip
	\noindent
	\emph{Step 4: comparison principle for the uniformly parabolic operator.}
	Finally set $U=-F$.  Equations \eqref{eq:scaled-flux} and
	\eqref{eq:source-small} give
	\[
		U_\theta-\alpha\tilde r^4(W^\nu U_\zeta)_\zeta
		\ge-C_\Theta k^\beta.
	\]
	After increasing $C_\Theta$ if necessary, define
	\[
		\underline U(\theta,\zeta)
		=(-J)e^{-\lambda_k\theta}\psi_k(\zeta)
		-C_\Theta k^\beta\theta.
	\]
	Equations \eqref{eq:proof-barrier-square}--\eqref{eq:proof-lambda-k}
	show that $\underline U$ is a subsolution on $I_k$.  At the two lateral
	endpoints, $\psi_k=0$, and \eqref{eq:F-range} gives
	$\underline U\le-F=U$.  Initially, $I_k$ lies in the flux plateau, so
	$U(0,\zeta)=-J$ and $\underline U(0,\zeta)\le U(0,\zeta)$.
	The operator is uniformly parabolic on $I_k$ by
	\eqref{eq:proof-W-observation}.  The comparison principle therefore gives,
	at the midpoint where $\psi_k(K_k)=1$,
	\begin{equation}
		\label{eq:proof-observation-prelimit}
		-F(\theta,K_k)
		\ge(-J)e^{-\lambda_k\theta}-C_\Theta k^\beta\theta.
	\end{equation}
	Since $\lambda_k\to0$ and $k^\beta\to0$, a final reduction of $k_0$
	makes the right-hand side of \eqref{eq:proof-observation-prelimit} at least
	$(-J)/2$, uniformly for $0\le\theta\le\Theta$.  This proves
	\eqref{eq:negative-observation-flux}.
\end{proof}

\subsection{The contradiction from finite capacity}

\begin{proposition}[Endpoint from finite reciprocal capacity]
	\label{prop:finite-capacity}
	There is a constant $\Theta_{\rm b}>0$, independent of $k$, such that,
	after reducing $k_0$ if necessary, the maximal lifespan with positive density
	satisfies
	\begin{equation}
		\label{eq:finite-endpoint}
		0<T_k^*\le \Theta_{\rm b}k^\beta
		=\Theta_{\rm b}k^{(\alpha-1)/\alpha}<\infty.
	\end{equation}
\end{proposition}

\begin{proof}
	We first record the affine capacity bounds used in the contradiction.  From
	\eqref{eq:W0},
	\[
		W_0(\zeta)
		=\Lambda_0+A_*\zeta+\frac1{R^4}\int_0^\zeta F_0(s)\dd s.
	\]
	For $L_0\le\zeta\le2K_k$, the plateau identity $F_0=J$ gives
	\[
		W_0(\zeta)=a_0\zeta+C_{{\rm e},k},
	\]
	where $C_{{\rm e},k}=\Lambda_0 + \frac1{R^4}\int_0^{L_0}\bigl(F_0(s)-J\bigr)\,ds$.  Since the correction $E_k$ is uniformly
	bounded on that interval, we may choose one constant $C_{\rm e}>0$,
	independent of small $k$, such that
	\begin{equation}
		\label{eq:capacity-W0-upper}
		W_0(\zeta)\le a_0\zeta+C_{\rm e},
		\qquad 0\le\zeta\le2K_k.
	\end{equation}

	Because $1/\alpha<1$, the reciprocal capacity gap is finite.  Indeed, with $\beta=1-1/\alpha$ and $x=a_0\zeta$,
	\begin{align}
		D_\infty
		&:=\int_0^\infty\left[
		(a_0\zeta)^{-1/\alpha}
		-(a_0\zeta+C_{\rm e})^{-1/\alpha}
		\right]\dd\zeta
	\notag\\
		&=\frac1{a_0}\lim_{L\to\infty}
		\left[
		\frac{x^\beta-(x+C_{\rm e})^\beta}{\beta}
		\right]_{x=0}^{x=L}
		=\frac{C_{\rm e}^{\,\beta}}{a_0\beta}<\infty.
	\label{eq:capacity-D-infinity}
	\end{align}
	Fix
	\begin{equation}
		\label{eq:capacity-Theta-b}
		\Theta_{\rm b}:=\frac{2(D_\infty+2)}{-J}.
	\end{equation}
	Proposition~\ref{prop:negative-observation-flux}, applied with
	$\Theta=\Theta_{\rm b}$, gives
	\begin{equation}
		\label{eq:capacity-flux-half}
		-F(\theta,K_k)\ge\frac{-J}{2},
		\qquad
		0\le\theta<
		\min\{\Theta_{\rm b},T_k^*/k^\beta\}.
	\end{equation}

	The scaled continuity equation, obtained directly from the first equation
	of \eqref{eq:radial-system}, is
	\begin{equation}
		\label{eq:capacity-scaled-continuity}
		W_\theta=\alpha W^\nu F_\zeta.
	\end{equation}
	Therefore, wherever the solution remains positive,
	\[
		\partial_\theta W^{-1/\alpha}
		=-\frac1\alpha W^{-1/\alpha-1}W_\theta
		=-F_\zeta.
	\]
	Since $F(\theta,0)=0$, integration over $[0,K_k]$ yields
	\begin{equation}
		\label{eq:capacity-identity}
		\frac{\dd}{\dd\theta}
		\int_0^{K_k}W(\theta,\zeta)^{-1/\alpha}\dd\zeta
		=-F(\theta,K_k).
	\end{equation}

	Suppose for contradiction that the positive branch survives beyond the
	target rescaled time $\Theta_{\rm b}$.  Integrating \eqref{eq:capacity-identity} gives
	\begin{equation}
		\label{eq:capacity-forced-increase}
		\begin{aligned}
			\int_0^{K_k}W(\Theta_b,\zeta)^{-1/\alpha}\dd\zeta-\int_0^{K_k}W_0(\zeta)^{-1/\alpha}\dd \zeta
			&=
			\int_0^{\Theta_{\rm b}}
			-F(\theta,K_k)\,d\theta\\
			&\ge
			\int_0^{\Theta_{\rm b}}\frac{-J}{2}\,d\theta
			=
			\frac{-J}{2}\Theta_{\rm b}
			=D_\infty+2.
		\end{aligned}
	\end{equation}

	On the other hand, let
	$a_k=a_0-\varepsilon_{k,\Theta_{\rm b}}$.  The intermediate slope estimate
	\eqref{eq:proof-slope-bound} gives $a_k\ge a_0/2$ for small $k$ and
	\[
		W(\theta,\zeta)
		=W(\theta,0)+\int_0^\zeta W_\eta(\theta,\eta)\dd\eta
		\ge a_k\zeta.
	\]
	Combining this with \eqref{eq:capacity-W0-upper}, we obtain
	\begin{equation}
		\label{eq:capacity-max-increase}
		\begin{aligned}
			&\int_0^{K_k}W(\theta,\zeta)^{-1/\alpha}\dd\zeta
			-\int_0^{K_k}W_0(\zeta)^{-1/\alpha}\dd\zeta \\
			&\quad\le
			\int_0^{K_k}
			\left[
			(a_k\zeta)^{-1/\alpha}
			-(a_0\zeta+C_{\rm e})^{-1/\alpha}
			\right]d\zeta \\
			&\quad=
			\int_0^{K_k}
			\left[
			(a_0\zeta)^{-1/\alpha}
			-(a_0\zeta+C_{\rm e})^{-1/\alpha}
			\right]+
			\left[
			(a_k\zeta)^{-1/\alpha}
			-(a_0\zeta)^{-1/\alpha}
			\right]d\zeta \\
			&\quad\le D_\infty
			+\frac{a_k^{-1/\alpha}-a_0^{-1/\alpha}}{\beta}K_k^\beta.
		\end{aligned}
	\end{equation}
	The map $a\mapsto a^{-1/\alpha}$ is Lipschitz near $a_0$, and hence the
	last term is bounded by
	\[
		C\varepsilon_{k,\Theta_{\rm b}}K_k^\beta=o(1)
	\]
	by \eqref{eq:proof-slope-capacity-error}.  It is at most $1$ for small $k$,
	so \eqref{eq:capacity-max-increase} permits an increase of at most
	$D_\infty+1$, contradicting
	\eqref{eq:capacity-forced-increase}.  This proves
	\eqref{eq:finite-endpoint}.

	We will also need the following consequence in the terminal argument:
	\begin{equation}
		\label{eq:density-loss}
		\liminf_{t\uparrow T_k^*}
		\min_{\overline{B_R}}\rho_k(t)=0.
	\end{equation}
	This is precisely the criterion for density loss
	\eqref{eq:bounded-time-density-loss} from
	Proposition~\ref{prop:bounded-time}, because
	\eqref{eq:finite-endpoint} has just shown that $T_k^*<\infty$.
\end{proof}

\subsection{Terminal trace and localization at the boundary}

For the exponents $p,q$ fixed in \eqref{eq:pq-range}, set
\begin{equation}
	\label{eq:terminal-exponents}
	s=\alpha-1+\frac1{2q},
	\qquad
	\kappa=\frac{1-\frac3{2q}}{\max\{1,s\}}>0,
	\qquad
	\vartheta=\frac{\kappa}{\kappa+1+\frac3{2p}}>0.
\end{equation}
Use the remaining freedom in \eqref{eq:chi-range} to impose also
\begin{equation}
	\label{eq:chi-localization}
	1-(\alpha-1)\vartheta<\chi<1.
\end{equation}
This is compatible with \eqref{eq:chi-range} because both lower endpoints are
strictly smaller than one.  No estimate used earlier for the observation
collar or capacity argument is changed.

\begin{proposition}[Terminal trace and localization at the boundary]
	\label{prop:terminal-wall}
	Choose $\chi$ satisfying both \eqref{eq:chi-range} and
	\eqref{eq:chi-localization}.  For every sufficiently small $k$, there is a
	unique nonnegative, nontrivial function
	$\rho_{k,*}\in C^{0,\kappa}(\overline{B_R})$ such that
	\begin{equation}
		\label{eq:terminal-wall-result}
		\rho_k(t)\longrightarrow\rho_{k,*}
		\quad\hbox{uniformly on }\overline{B_R}
		\quad\hbox{as }t\uparrow T_k^*,
		\qquad
		\{r\in[0,R]:\rho_{k,*}(r)=0\}=\{R\}.
	\end{equation}
\end{proposition}

\begin{proof}
	We first obtain the compactness needed to define the terminal trace.  Since
	$T_k^*\le\Theta_{\rm b}k^\beta<1$ for small $k$, the intermediate estimate
	\eqref{eq:uniform-HMZ-family-estimates} may be used with the fixed physical
	time bound $T_0=1$.  Set
	\[
		g_k:=\rho_k^s,
		\qquad
		\eta:=1-\frac3{2q}>0.
	\]
	The density upper bound in
	\eqref{eq:uniform-HMZ-family-estimates} and the finiteness of $B_R$ give
	\[
		\norm{g_k(t)}_{L^{2q}(B_R)}
		\le \abs{B_R}^{1/(2q)}
		\norm{\rho_k(t)}_{L^\infty(B_R)}^s
		\le C.
	\]
	Combining this with the gradient bound in the same estimate yields
	\[
		\sup_{0<k<k_0}\ \sup_{0\le t<T_k^*}
		\norm{g_k(t)}_{W^{1,2q}(B_R)}\le C.
	\]
	Since $2q>3$, Morrey's inequality in three dimensions now gives
	\[
		\abs{g_k(t,x)-g_k(t,y)}
		\le C\abs{x-y}^{\eta},
		\qquad x,y\in\overline{B_R}.
	\]
	It remains to pass from $g_k=\rho_k^s$ to $\rho_k$.  If $0<s\le1$, then
	$z\mapsto z^{1/s}$ is Lipschitz on the common bounded range of the $g_k$,
	so
	\[
		\abs{g_k(t,x)^{1/s}-g_k(t,y)^{1/s}}
		\le C\abs{g_k(t,x)-g_k(t,y)}.
	\]
	If $s>1$, the elementary inequality
	\[
		\abs{A^{1/s}-B^{1/s}}
		\le\abs{A-B}^{1/s},
		\qquad A,B\ge0,
	\]
	instead gives the exponent $\eta/s$.  Thus, in both cases,
	\[
		\abs{\rho_k(t,x)-\rho_k(t,y)}
		\le C\abs{x-y}^{\eta/\max\{1,s\}}
		=C\abs{x-y}^{\kappa}.
	\]
	Together with the uniform density upper bound, this proves
	\begin{equation}
		\label{eq:terminal-space-holder}
		\sup_{0<k<k_0}\ \sup_{0\le t<T_k^*}
		\norm{\rho_k(t)}_{C^{0,\kappa}(\overline{B_R})}\le C.
	\end{equation}
	The density upper bound and the $2p$ moment in
	\eqref{eq:uniform-HMZ-family-estimates} give
	\[
		\begin{aligned}
		\norm{\rho_k\boldsymbol u_k}_{L^{2p}(B_R)}^{2p}
		&=\int_{B_R}\rho_k^{2p}\abs{\boldsymbol u_k}^{2p}\dd x\\
		&=\int_{B_R}\rho_k^{2p-1}
		\bigl(\rho_k\abs{\boldsymbol u_k}^{2p}\bigr)\dd x\\
		&\le\norm{\rho_k}_{L^\infty}^{2p-1}
		\int_{B_R}\rho_k\abs{\boldsymbol u_k}^{2p}\dd x\\
		&\le C.
		\end{aligned}
	\]
	Let $(2p)'=2p/(2p-1)$.  For
	$\varphi\in W^{1,(2p)'}(B_R)$, the continuity equation gives
	\[
		\begin{aligned}
		\abs{\langle\partial_t\rho_k,\varphi\rangle}
		&=\abs{\int_{B_R}\rho_k\boldsymbol u_k
		\mathbin{\cdot}\nabla\varphi\dd x}\\
		&\le\norm{\rho_k\boldsymbol u_k}_{L^{2p}(B_R)}
		\norm{\nabla\varphi}_{L^{(2p)'}(B_R)}
		\le C\norm{\varphi}_{W^{1,(2p)'}(B_R)}.
		\end{aligned}
	\]
	Hence
	\[
		\norm{\partial_t\rho_k(t)}_{W^{-1,2p}(B_R)}\le C
	\]
	uniformly for $0\le t<T_k^*$.  Integrating this estimate in time, we obtain, for $0\le s_1<t<T_k^*$,
	\begin{equation}
		\label{eq:terminal-negative-time-bound}
		\norm{\rho_k(t)-\rho_k(s_1)}_{W^{-1,2p}(B_R)}
		\le C\abs{t-s_1}.
	\end{equation}

	We next interpolate the spatial estimate
	\eqref{eq:terminal-space-holder} with the time estimate in a negative norm
	\eqref{eq:terminal-negative-time-bound}.  Put
	\[
		h:=\rho_k(t)-\rho_k(s_1),
		\qquad
		\tau:=\abs{t-s_1}.
	\]
	The case $\tau=0$ is immediate, so suppose $\tau>0$.  By
	\eqref{eq:terminal-space-holder},
	\[
		\norm{h}_{C^{0,\kappa}(\overline{B_R})}\le C.
	\]
	Fix a standard nonnegative mollifier
	$\phi\in C_c^\infty(\mathbb R^3)$ with $\int_{\mathbb R^3}\phi=1$, and set
	\[
		\phi_\ell(x):=\ell^{-3}\phi(x/\ell),
		\qquad 0<\ell<1.
	\]
	After applying a fixed extension operator for the ball, bounded in the two
	norms used below, define $h_\ell:=\phi_\ell*h$.  The H\"older estimate gives
	\[
		\begin{aligned}
		\abs{h(x)-h_\ell(x)}
		&\le \int_{\mathbb R^3}\phi_\ell(y)
		\abs{h(x)-h(x-y)}\dd y\\
		&\le C\int_{\mathbb R^3}\phi_\ell(y)\abs y^\kappa\dd y\\
		&=C\ell^\kappa
		\int_{\mathbb R^3}\phi(z)\abs z^\kappa\dd z
		\le C\ell^\kappa.
		\end{aligned}
	\]
	Consequently,
	\[
		\norm{h-h_\ell}_{L^\infty(B_R)}\le C\ell^\kappa.
	\]

	To estimate the smoothed part, let $r=2p$ and
	$r'=(2p)'=2p/(2p-1)$.  For each fixed $x$,
	\[
		h_\ell(x)=\left\langle h,\phi_\ell(x-\cdot)\right\rangle.
	\]
	The exact three-dimensional scaling of the mollifier is
	\[
		\norm{\phi_\ell}_{L^{r'}(\mathbb R^3)}
		=\ell^{-3+3/r'}\norm\phi_{L^{r'}}
		=\ell^{-3/r}\norm\phi_{L^{r'}},
	\]
	while
	\[
		\norm{\nabla\phi_\ell}_{L^{r'}(\mathbb R^3)}
		=\ell^{-4+3/r'}\norm{\nabla\phi}_{L^{r'}}
		=\ell^{-1-3/r}\norm{\nabla\phi}_{L^{r'}}.
	\]
	Since $0<\ell<1$, the gradient term dominates, and hence
	\[
		\norm{\phi_\ell(x-\,\cdot)}_{W^{1,(2p)'}}
		\le C\ell^{-1-3/(2p)}.
	\]
	Duality and \eqref{eq:terminal-negative-time-bound} now imply
	\[
		\begin{aligned}
		\abs{h_\ell(x)}
		&\le \norm h_{W^{-1,2p}}
		\norm{\phi_\ell(x-\,\cdot)}_{W^{1,(2p)'}}\\
		&\le C\tau\ell^{-1-3/(2p)}.
		\end{aligned}
	\]
	Taking the supremum in $x$ and combining the two parts yields
	\[
		\norm h_{L^\infty(B_R)}
		\le C\left(
		\ell^\kappa+\tau\ell^{-1-3/(2p)}
		\right).
	\]
	Because $T_k^*<1$, one has $0<\tau<1$.  Choose the scale by balancing the
	two terms:
	\[
		\ell^{\kappa+1+3/(2p)}=\tau,
		\qquad
		\ell=\tau^{1/(\kappa+1+3/(2p))}\in(0,1).
	\]
	For this choice,
	\[
		\ell^\kappa
		=\tau^{\kappa/(\kappa+1+3/(2p))},
		\qquad
		\tau\ell^{-1-3/(2p)}
		=\tau^{\kappa/(\kappa+1+3/(2p))}.
	\]
	By the definition of $\vartheta$ in \eqref{eq:terminal-exponents}, this
	proves
	\begin{equation}
		\label{eq:terminal-time-holder}
		\norm{\rho_k(t)-\rho_k(s_1)}_{L^\infty(B_R)}
		\le C\abs{t-s_1}^{\vartheta}.
	\end{equation}
	Thus $\rho_k(t)$ is Cauchy in $L^\infty$ as $t\uparrow T_k^*$ and converges
	uniformly to a unique
	$\rho_{k,*}\in C^{0,\kappa}(\overline{B_R})$.  Conservation of mass and
	$M_k\ge M_->0$ show that this terminal profile is nontrivial, while
	\eqref{eq:density-loss} shows that it has at least one zero.

	It remains to locate the zero.  Set
	\begin{equation}
		\label{eq:terminal-delta-b}
		\delta_{{\rm b},k}:=2kK_k=2k^{1-\chi}.
	\end{equation}
	On the initial affine collar, $(W_0)_\zeta\ge a_0$ gives
	\[
		\rho_{0,k}(r_{0,k}(\delta))^\alpha
		=kW_0(\delta/k)\ge a_0\delta.
	\]
	At $\delta\ge\delta_{{\rm b},k}$ this is bounded below by
	$2a_0k^{1-\chi}$, and the center extension has a positive lower bound
	independent of small $k$.  Consequently, there is a $c_0>0$, independent of
	$k$, such that
	\begin{equation}
		\label{eq:terminal-initial-complement-lower}
		\rho_{0,k}(r_{0,k}(\delta))
		\ge c_0k^{(1-\chi)/\alpha},
		\qquad
		\delta_{{\rm b},k}\le\delta\le M_k.
	\end{equation}

	At a fixed label in inward mass coordinates, $r_t=u$.  The estimate for the transverse strain rate functions, the radius bound $r\le R$, and \eqref{eq:finite-endpoint} give
	\begin{equation}
		\label{eq:terminal-particle-motion}
		\abs{r(t,\delta)-r_{0,k}(\delta)}
		\le Ct\le Ck^\beta.
	\end{equation}
	Insert the intermediate value $\rho_{0,k}(r(t,\delta))$.  Using
	\eqref{eq:terminal-time-holder}, then
	\eqref{eq:terminal-space-holder}, and finally
	\eqref{eq:terminal-particle-motion}, we obtain
	\begin{align}
		&\abs{\rho_k(t,r(t,\delta))
		-\rho_{0,k}(r_{0,k}(\delta))}
	\notag\\
		&\qquad\le Ck^{\beta\vartheta}+Ck^{\beta\kappa}
		\le Ck^{\beta\vartheta}
		=o\left(k^{(1-\chi)/\alpha}\right).
	\label{eq:terminal-complement-error}
	\end{align}
	Here $\vartheta<\kappa$, and
	\[
		\beta\vartheta>\frac{1-\chi}{\alpha}
		\quad\Longleftrightarrow\quad
		\chi>1-(\alpha-1)\vartheta,
	\]
	which is exactly \eqref{eq:chi-localization}.  Hence, for every sufficiently
	small fixed $k$,
	\begin{equation}
		\label{eq:terminal-complement-positive}
		\rho_k(t,r(t,\delta))
		\ge\frac{c_0}{2}k^{(1-\chi)/\alpha},
		\qquad
		\delta_{{\rm b},k}\le\delta\le M_k,
		\quad 0\le t<T_k^*.
	\end{equation}

	On the moving boundary collar $0\le\delta\le\delta_{{\rm b},k}$, the
	intermediate slope estimate \eqref{eq:proof-slope-bound}, used for
	$0\le\theta\le\Theta_{\rm b}$, gives
	\begin{equation}
		\label{eq:terminal-physical-slope}
		(\rho_k^\alpha)_\delta=W_\zeta\ge\frac{a_0}{2}>0.
	\end{equation}
	Thus the boundary is the density minimum throughout this collar.  Since
	$\delta_r=-\rho r^2$, one also has the exact identity
	\[
		\partial_r\rho^{\alpha-1}
		=-\frac{\alpha-1}{\alpha}r^2W_\zeta.
	\]
	Because $r\ge R/2$ on the collar, integration from $r$ to $R$ gives
	\begin{equation}
		\label{eq:terminal-collar-interior-lower}
		\rho_k(t,r)^{\alpha-1}
		-\rho_k(t,R)^{\alpha-1}
		\ge c(R-r).
	\end{equation}

	The complement bound \eqref{eq:terminal-complement-positive} and the collar monotonicity \eqref{eq:terminal-physical-slope} show that the density loss in \eqref{eq:density-loss} can occur only at the boundary.  Uniform terminal convergence therefore gives $\rho_{k,*}(R)=0$.  For every fixed $r<R$, the complement/collar dichotomy and
	\eqref{eq:terminal-collar-interior-lower} give
	\[
		\rho_k(t,r)\ge
		\min\left\{
		\frac{c_0}{2}k^{(1-\chi)/\alpha},
		\bigl(c(R-r)\bigr)^{1/(\alpha-1)}
		\right\}>0.
	\]
	Passing to the terminal limit proves
	$\rho_{k,*}(r)>0$ for every $r<R$.  Hence the terminal zero set is exactly
	$\{R\}$.

	As a final intermediate byproduct, the effective boundary velocity retains its strictly negative sign.  Indeed, no slip gives $u_k(t,R)=0$, so the equation for the effective velocity at the boundary reduces to
	\[
		\frac{\dd}{\dd t}v_k(t,R)
		+F_1(\rho_k(t,R))v_k(t,R)=0.
	\]
	Using $v_{0,k}(R)=-R^2A_*$ from \eqref{eq:wall-values}, we obtain
	\[
		v_k(t,R)
		=-R^2A_*
		\exp\left(-\int_0^tF_1(\rho_k(s,R))\dd s\right)<0,
		\qquad 0\le t<T_k^*.
	\]
	This completes \eqref{eq:terminal-wall-result}.
\end{proof}

\begin{proof}[Proof of Theorem~\ref{thm:bad-wall}]
	Choose the family of smooth positive compatible data from
	Lemma~\ref{lem:data}.  Proposition~\ref{prop:local-restart} produces its
	maximal positive radial classical branch, and
	\eqref{eq:wall-values} gives $v_{0,k}(R)<0$.  After choosing $k_0$ small
	enough for all preceding uniform estimates, Proposition~\ref{prop:finite-capacity}
	yields
	\[
	0<T_k^*\le \Theta_{\rm b}k^{(\alpha-1)/\alpha}<\infty.
	\]
	Finally, Proposition~\ref{prop:terminal-wall} supplies the unique nontrivial
	terminal density profile, uniform convergence to that profile, and the exact
	zero set $\{R\}$.  Taking $C=\Theta_{\rm b}$ completes the proof.
\end{proof}

\section*{Acknowledgments}
X Huang is partially supported by Chinese Academy of Sciences Project for Young Scientists in Basic Research (Grant No. YSBR-031), National Natural Science Foundation of China (Grant Nos. 12494542, 11688101). B Li is supported in part by National Key R\&D Program of China (No.~2024YFA\allowbreak 1013303).

\par\medskip
\noindent\textbf{Data availability statement.}
Data sharing is not applicable to this article.
\par

\medskip
\noindent\textbf{Conflict of interest.}
The authors declare that they have no conflict of interest.
\par

\end{document}